\documentclass[11pt]{article}
\usepackage{epigamath}

\usepackage[notext]{kpfonts}
\usepackage{baskervald}

\setpapertype{A4}

\usepackage[english]{babel}

\usepackage[dvips]{graphicx}     

\removelength{1cm}

\title{Chern classes of automorphic vector bundles, II}
\titlemark{Chern classes of automorphic vector bundles, II}
\author{H\'el\`ene Esnault and Michael Harris}
\authoraddresses{
\authordata{H\'el\`ene Esnault}{\firstname{H\'el\`ene} \lastname{Esnault}\\
\institution{Freie Universit\"at Berlin, Mathematik, Arnimallee 3, 14195 Berlin, Germany}\\
\email{esnault@math.fu-berlin.de}} \\
\authordata{Michael Harris}{\firstname{Michael} \lastname{Harris}\\
\institution{Department of Mathematics,
Columbia University,   2990 Broadway
New York, New York 10027, USA}\\
\email{harris@math.columbia.edu}}
}
\authormark{H. Esnault and M. Harris}
\date{\vspace{-5ex}} 
\journal{\'Epijournal de G\'eom\'etrie Alg\'ebrique} 
\acceptation{Received by the Editors on July 25, 2018, and in final form on July 25, 2019. \\ Accepted on September 2, 2019.}

\newcommand{\articlereference}{Volume 3 (2019), Article Nr. 14}

\acknowledgement{The research leading to these results has received funding from the European Research Council under the European Community's Seventh Framework Programme (FP7/2007-2013) / ERC Advanced Grants  22224001 and 290766 (AAMOT), the Einstein Foundation,    the NSF grants DMS-1404769 and DMS-1701651.}

\usepackage[all]{xy}

\allowdisplaybreaks

\newdimen\origiwspc
\origiwspc=\fontdimen2\font

\usepackage{float}

\numberwithin{equation}{section}

\makeatletter

\renewcommand{\p@equation}{\arabic{section}.\arabic{equation}\expandafter\@gobble}
\makeatother

\usepackage{amscd,amsmath,amssymb,amsfonts}
\usepackage{appendix}

\newtheorem{thm}[equation]{Theorem}

\newtheorem{lem}[equation]{Lemma}
\newtheorem{cor}[equation]{Corollary}

\newtheorem{prop}[equation]{Proposition}

\newtheorem{hyp}[equation]{Hypothesis}

\newtheorem{nota}[equation]{Notation} 
\newtheorem{defn}[equation]{Definition}
\newtheorem{disc}[equation]{Discussion}

\newtheorem{rmk}[equation]{Remark}
\newtheorem{rmks}[equation]{Remarks}

\newcommand{\isoarrow}{{~\overset\sim\longrightarrow~}}

\newcommand{\sS}{{\mathcal S}}

\newcommand{\g}{{\gamma}}

\newcommand{\CE}{{\mathcal{E}}}
\newcommand{\CM}{{\mathcal{M}}}
\newcommand{\CS}{{\mathcal{S}}}

\newcommand{\CA}{{\mathcal{A}}}

\newcommand{\CW}{{\mathcal{W}}}
\newcommand{\CX}{{\mathcal{X}}}

\newcommand{\fg}{{\mathfrak{g}}}

\newcommand{\ra}{{~\rightarrow~}}
\newcommand{\Ss}{{\Sigma}}

\newcommand{\CH}{{\mathcal{H}}}
\newcommand{\hX}{{\hat{X}}}

\newcommand{\QQ}{{\mathbb Q}}

\newcommand{\Qbar}{{\overline{\mathbb Q}}}

\newcommand{\Qpb}{{\overline{\mathbb Q}_{p}}}

\newcommand{\RR}{{\mathbb R}}
\newcommand{\ad}{{\mathbf A}}
\newcommand{\af}{{{\mathbf A}_f}}

\newcommand{\CC}{{\mathbb C}}
\newcommand{\PP}{{\mathbb P}}

\newcommand{\Qp}{{{\mathbb Q}_p}}
\newcommand{\Zp}{{{\mathbb Z}_p}}

\newcommand{\Vect}{{\rm Vect}}  
\newcommand{\Rep}{{\rm Rep}} 

\newcommand{\GL}{\mathrm{GL}}
\newcommand{\SO}{\mathrm{SO}}
\newcommand{\PGL}{\mathrm{PGL}}
\newcommand{\SP}{\mathrm{Sp}}
\newcommand{\GSp}{\mathrm{GSp}}
\newcommand{\Isom}{\mathrm{Isom}}
\newcommand{\Proj}{\mathrm{Proj}}
\newcommand{\St}{\mathrm{St}}
\newcommand{\ml}[2]{\begin{multline}\label{#1}#2 \end{multline}} 
\newcommand{\ga}[2]{\begin{gather}\label{#1}#2 \end{gather}} 

\begin{document}


\maketitle



\begin{prelims}

\vspace{-0.55cm}

\def\abstractname{Abstract}
\abstract{We prove that  the $\ell$-adic Chern classes of  canonical extensions of automorphic vector bundles, over toroidal compactifications of Shimura varieties of Hodge type over $\bar{ \mathbb{Q}}_p$,  descend to classes in the $\ell$-adic cohomology of the minimal compactifications.  These are invariant under the Galois group of the $p$-adic field above which the variety and the bundle are defined.}




\languagesection{Fran\c{c}ais}{%

\vspace{-0.05cm}
{\bf Titre. Classes de Chern des fibr\'es vectoriels automorphes, II} \commentskip {\bf R\'esum\'e.} Nous d\'emontrons que, sur les compactifications toro\"{\i}dales des vari\'et\'es de Shimura de type Hodge sur $\bar{ \mathbb{Q}}_p$, les classes de Chern $\ell$-adiques des extensions canoniques des fibr\'es vectoriels automorphes descendent en des classes de cohomologie $\ell$-adique sur les compactifications minimales. Elles 
sont invariantes sous l'action du groupe de Galois du corps $p$-adique sur lequel la vari\'et\'e et le fibr\'e sont d\'efinis.}

\end{prelims}


\newpage

\setcounter{tocdepth}{1} \tableofcontents

\section*{Introduction} \label{sec:intro}
\addcontentsline{toc}{section}{Introduction}
Let $S$ be a Shimura variety. It is defined over a number field \cite[Corollaire 5.5]{ Del71}, called the reflex field $E$,  and carries a family of automorphic vector bundles $\mathcal{E}$, defined (collectively) over the same number field (\cite[Theorem 4.8]{Har85}; the rationality conventions of \cite[\S 1]{EH17} are recalled in \S \ref{sec:autobdl} below). The Shimura variety has a minimal compactification $S\hookrightarrow  S^{\rm min}$ which in dimension $\ge 2$  is singular. This is minimal in the sense that each of the toroidal compactifications $S^{\rm tor}_\Sigma$, subject to the extra choice of a fan $\Sigma$, admits a birational proper morphism
$S^{\rm tor}_\Sigma\to S^{\rm min}$, which is an isomorphism on $S$ and (under additional hypotheses on the choice of fan and level subgroup, that we will make systematically) is a desingularization of $S^{\rm min}$. The automorphic bundle $\mathcal{E}$ admits  a canonical extension $\mathcal{E}^{\rm can}$ on $S^{\rm tor}_\Sigma$ constructed by Mumford \cite[Theorem~3.1]{Mum77} and more generally in \cite{Har89}; $\mathcal{E}^{\rm can}$  is locally free and defined over a finite extension of the reflex field. On $S^{\rm min},$ on the other hand, there is generally no locally free extension of $\mathcal{E}$, except in dimension $1$ when  $S^{\rm tor}_\Sigma \to S^{\rm min}$ is an isomorphism. In \cite{EH17} we studied the continuous $\ell$-adic Chern classes of $\mathcal{E}$ in case $S$ is proper. In particular, we proved that the higher Chern classes in continuous $\ell$-adic cohomology die if $\mathcal{E}$ is flat \cite[Theorem~0.2]{EH17}. A key ingredient in the proof is the study of the action of Hecke algebra on continuous $\ell$-adic cohomology, and the fact that Chern classes lie in the eigenspace for the volume character \cite[Corollary~1.18]{EH17}. When $S$ is not proper, the Hecke algebra does not act on any cohomology of $S^{\rm tor}_\Sigma$ . Thus we cannot apply the methods of {\it loc.cit.} to prove the analogous theorem  for the canonical extensions of flat automorphic bundles on $S^{\rm tor}$. 

\medskip

On the other hand, the Hecke algebra does act on the $\ell$-adic cohomology of $S^{\rm min}$.  In addition, over the field of complex numbers $\mathbb{C}$, Goresky-Pardon \cite[Theorem]{GP02} used 
explicit estimates on differential forms to construct
classes  $c_n(\mathcal{E})^{GP}\in H^{2n}(S^{\rm min}, \mathbb{C})$ 
which descend the Chern classes of $\mathcal{E}^{\rm can}$ in Betti cohomology $H^{2n}(S^{\rm tor}_\Sigma, \mathbb{Q}(n))$.  They do not come from classes in $H^{2n}(S^{\rm min}, \mathbb{Q}(n))$; precisely this fact enabled Looijenga \cite[Theorem~5.1]{Loo17} to construct mixed Tate extensions in the Hodge category on the Siegel modular variety (see also the  result announced by Nair in \cite[\S0.4]{Nai14}). 

Our main theorem gives an $\ell$-adic version of the Goresky-Pardon construction. 
\begin{thm}[Theorem~\ref{chernp}, Theorem~\ref{pullback_Chern}] \label{thm:main}
Let $S$ be a Shimura variety of Hodge type.  Let $\ell$ be a prime number,  $E_w$ the completion of $E$ at a place $w$ dividing the prime $p$ different from $\ell$, and $E_w\hookrightarrow \bar{ \mathbb{Q}}_p$ be an algebraic closure. Then the Chern classes of $\mathcal{E}^{\rm can}$ in $H^{2n}(S^{\rm tor}_{\Sigma, \bar {\mathbb Q} _p}, \mathbb{Q}_\ell(n))$ descend to well defined classes $c_n(\mathcal{E})$ in $H^{2n}(S_{\bar{\mathbb{Q}}_p} ^{\rm min}, \mathbb{Q}_\ell(n))$.  The $c_n(\mathcal{E})$ are contravariant for the change of level  and of Shimura data,  verify the Whitney product formula, lie in the eigenspace for the volume character of the Hecke algebra,
and are invariant under the action of the Galois group of $E_w$ on $H^{2n}(S_{\bar{\mathbb{Q}}_p} ^{\rm min}, \mathbb{Q}_\ell(n))$.  Their pullback  under
$$H^{2n}\left(S_{\bar{\mathbb{Q}}_p} ^{\rm min}, \mathbb{Q}_\ell(n)\right) \to 
H^{2n}\left(S_{\bar{\mathbb{Q}}_p} ^{\rm tor}, \mathbb{Q}_\ell(n)\right) = H^{2n}\left(S_{\Sigma, \bar{\mathbb{Q}}} ^{\rm tor}, \mathbb{Q}_\ell(n)\right)$$
are the Chern classes of $\mathcal{E}^{\rm can}$ which are invariant under the Galois group of $E$.

\end{thm}

Here, as in \cite[Lemma 1.15]{EH17}, the {\it volume character} is the character by which the Hecke algebra acts on the constant functions.  The theorem can easily be extended to Shimura varieties of abelian type (see Remarks \ref{abeliantype}) but we leave the precise statement to the reader.

As the classes of  $\mathcal{E}^{\rm can} $ are even defined in continuous $\ell$-adic cohomology of $S^{\rm tor}_{ \bar{\mathbb{Q}}}$, where $E\hookrightarrow \bar {\mathbb Q}$ is an algebraic closure,  they are in particular invariant under the Galois group of $E$.  The classes we construct  rely on $p$-adic geometry and have no reason a priori not to depend on the chosen $p$, let alone to lie in $H^{2n}(S_{\bar{\mathbb{Q}}} ^{\rm min}, \mathbb{Q}_\ell(n))^{G_E}$ via the isomorphism
$H^{2n}(S_{\bar{\mathbb{Q}}_p} ^{\rm min}, \mathbb{Q}_\ell(n)) = H^{2n}(S_{\bar{\mathbb{Q}}} ^{\rm min}, \mathbb{Q}_\ell(n))$.

\medskip

We now describe the method of proof. The main tool used is the existence of a perfectoid space  $P^{\rm min}$ above\footnote{We are using a simplified notation here for the space denoted ${}_{K^p}\CS(G,X)^{\rm min}$ in the body of the text.  Automorphic vector bundles are also denoted $[\CE]$ rather than $\CE$, and the superscript $^{\rm tor}$ is absent in most cases.} $S^{\rm min}$ which has been constructed by Peter Scholze in \cite[Theorem~4.1.1]{Sch15} for Shimura varieties of Hodge type, together with a Hodge-Tate period map  $\pi_{HT}: P^{\rm min}\to \hat{\mathcal{X}}$ 
(see \cite[Theorem~2.3]{CS17} for the final form)   with values in the adic space of the compact dual variety defined over the completion of $\bar{\mathbb{Q}}_p$.  
 This allows us to define  {\it vector bundles} $\pi_{HT}^*(  \mathcal{E}  )$  on $P^{\rm min}$  where $\mathcal{E}$ are equivariant vector bundles on $\hat{\mathcal{X}}$ defining the automorphic bundles.  So while the bundles $\mathcal{E}$ do not extend to $S^{\rm min}$, they do on  Scholze's limit space $P^{\rm min}$. 
 
On the other hand, suppose $S$ is the Siegel modular variety.  Then Pilloni and Stroh \cite[Corollaire~1.6]{PS16} have constructed a perfectoid space  $P^{\rm tor}$ 
above $S^{\rm tor}$.   On $P^{\rm tor}$ one has the pullback of the bundles  $\pi_{HT}^*(\mathcal{E} )$ and the pullback via the tower defining the perfectoid space of the canonical extensions $\mathcal{E}^{\rm can}$.  Theorem~\ref{compar2} asserts that the two pullbacks on $P^{\rm tor}$ are the same.  This argument generalizes to Shimura varieties satisfying a certain technical Hypothesis \ref{projtor}.  We expect Hypothesis \ref{projtor} to hold in general; however, for general Shimura varieties of Hodge type we provide a more indirect route in \S 3 to the same comparison (Theorem ~\ref{compar3}),  following a suggestion of Bhargav Bhatt.

Descent for cohomology with torsion coefficients enables  one to construct the classes as claimed in Theorem~\ref{thm:main}, see Appendix~\ref{sec:app}.

\medskip
The classes $c_n(\mathcal{E}) \in  H^{2n}(S_{\bar{\mathbb{Q}}_p} ^{\rm min}, \mathbb{Q}_\ell(n))$ map in particular to well-defined classes in $\ell$-adic intersection cohomology  $IH^{2n}(S_{\bar{\mathbb{Q}}_p} ^{\rm min}, \mathbb{Q}_\ell(n))$. On intersection cohomology  of $S^{\rm min}$ over $\bar{\mathbb{Q}}$, we identify the eigenspace under the volume character with the cohomology of the compact dual $\hat{X}$ of $S$ (Proposition~ \ref{Chern}).  If the classes in $IH^{2n}(S_{\bar{\mathbb{Q}}_p} ^{\rm min}, \mathbb{Q}_\ell(n))$, identified with 
$IH^{2n}(S_{\bar{\mathbb{Q}}} ^{\rm min}, \mathbb{Q}_\ell(n))$,
 descended to continuous $\ell$-adic intersection cohomology, we could try to apply the method developed in \cite{EH17} to show that the classes of $\mathcal{E}^{\rm can}$  in continuous $\ell$-adic cohomology on $S^{\rm tor}$ over $E$ die when $\mathcal{E}$ if flat. 
  This would be in accordance with \cite[Theorem~1.1]{EV02} where it is shown that the higher Chern classes of $\mathcal{E}^{\rm can}$ in the rational Chow groups on the Siegel modular variety vanish when $\mathcal{E}$ is flat. 
However, we are not  yet able to work with continuous $\ell$-adic intersection cohomology. We prove in Lemma~\ref{lem:flat}  the vanishing  $c_n(\mathcal{E}) =0$  for all $n>0$
 when $\mathcal{E}$ is flat.

\medskip

While we were writing the present note, it was brought to our attention that Nair  in  the unpublished manuscript  \cite{Nai14}  independently mentioned the possibility of using Scholze's Hodge-Tate morphism to construct Chern classes in the cohomology of minimal compactifications, see {\it loc.~cit.}~\S0.4. 
 
 \medskip
 {\it Acknowledgements}: This note is based on Peter Scholze's construction of a perfectoid space above the minimal compactification of Shimura varieties of Hodge type together with the Hodge-Tate period map.
 We thank him profoundly  for the numerous discussions on his theory, for his precise and generous answers to our questions. We thank Vincent Pilloni for explaining to us the content of \cite[Corollaire~1.6]{PS16} which  plays an important role in our note, and for pointing out some oversights in an earlier version of this note.  We thank David Hansen for providing the keys to navigating the rapidly expanding literature on perfectoid geometry; we thank Hansen again, as well as Laurent Fargues, for their patient and detailed explanations that allowed us to correct misunderstandings that one of the referees pointed out in an earlier version of this paper.  We thank Mark Goresky  for helpful  discussions,  Simon Pepin Lehalleur who brought the reference  \cite[\S0.4]{Nai14} to our attention, and Ana Caraiani, who helped us to understand the issues that forced us to take an unexpected detour in the proof of Proposition \ref{CY}.   We are deeply indebted to the two referees who helped us to improve our article with their precise and helpful comments. We apologize  for the work we unnecessarily gave them through our numerous  typos.   Finally, we are  grateful to Bhargav Bhatt for a number of  crucial suggestions, notably for the construction of $Z_s$ in \eqref{Zs}.

\section{Generalities on automorphic bundles } \label{sec:autobdl}

If $G$ is a reductive algebraic group over $\QQ$, by an {\it admissible irreducible representation} of $G(\ad)$ we mean an irreducible admissible $(\fg,K)\times G(\af)$-module, where $\fg$ is the complexified Lie algebra of $G$, $K \subset G(\RR)$ is a connected subgroup generated by the center of $G(\RR)$ and a maximal compact connected subgroup,   $G(\ad)$ ($G(\af )$)    is the group of (finite) ad\`eles of $G$.
If $\pi$ is such a representation then we 
 write
 $$\pi \simeq \pi_{\infty} \otimes \pi_f$$
 where $\pi_{\infty}$ is an irreducible admissible $(\fg,K)$-module and $\pi_f$ is an irreducible admissible representation of $G(\af)$.

Let $(G,X)$ be a Shimura datum:  $G$ is a reductive algebraic group over $\QQ$, and $X$ is a (finite union of) hermitian symmetric spaces, endowed with a transitive action of $G(\RR)$ satisfying a familiar list of axioms.   The compact dual $\hX$ of $X$ is a (projective) flag variety for $G$.  Thus one can speak of the category $\Vect_G(\hX)$ of $G$-equivariant vector bundles on $\hX$; the choice of a base point $h \in X \subset \hX$ determines an equivalence of categories
$\Rep(P_h) \simeq \Vect_G(\hX),$
where $P_h \subset G$ is the stabilizer of $h$, a maximal parabolic subgroup of $G$, and  for any algebraic group $H$,  $\Rep(H)$ denotes the tensor category of its representations. Let $K_h$ be the Levi quotient of $P_h$; the group $K_{h,\infty} := K_h(\CC)\cap G(\RR)$ can be identified with a maximal Zariski connected  subgroup of $G(\RR)$ that is compact modulo the center of $G$.  If $K \subset G(\af)$ is an open compact subgroup, we let ${}_KS(G,X)$ denote the  Shimura variety attached to $(G,X)$ at level $K$; it has a canonical model over the  {\it  reflex field} $E(G,X)$ which is a number field (see \cite[Corollaire~5.5]{ Del71}).  We always assume that $K$ is {\it neat}; then ${}_KS(G,X)$ is a smooth quasi-projective variety.

Furthermore, we use the notation and conventions of \cite[\S1.1--1.2]{EH17},
specifically  $\mathcal{E} \in {\rm Vect}(\hat X)$ the notation for the  vector  bundle on the compact dual,  $[\mathcal{E}]_K \in 
{\rm Vect}({}_KS(G,X))$  for the automorphic vector bundle associated to the underlying representation of the compact (mod center) group $K_{h,\infty}  \subset G(\mathbb{R})$, the stabilizer of a chosen point $h\in X$.  As in \cite{EH17}, we always assume that $K_h$ is defined over a CM field and that every irreducible representation of $K_h$ has a model rational over the CM field $E'_h$; in particular, $P_h$ is also defined over $E'_h$.  For the purposes of constructing Chern classes, we need only consider semisimple representations of $P_h$, which necessarily factor through representations of $K_h$.

The action of the Hecke algebra  $\mathcal{H}_K=\{T(g), g \in G(\af)\}$ 
is recalled in Section~\ref{sec:chern}. 
The Chern character  
$$ch_\hX:  \Vect_G(\hX) \ra CH(\hX)_\QQ  \to \oplus_i H^{2i}\left(\hX,\QQ(i)\right) $$
 induces  isomorphisms
\begin{equation}\label{01} 
K_0\left(\Rep(K_h)\right)\otimes_{K_0\left(\Rep(G)\right)} \QQ  \isoarrow CH(\hX)_{\QQ} \isoarrow \oplus_i H^{2i}\left(\hX,\QQ(i)\right)
\end{equation}
and analogously in $\ell$-adic cohomology.
On the other hand, the construction of automorphic vector bundles gives rise to a homomorphism
\begin{equation}\label{chK} 
ch_K:\left\{\begin{array}{rcl} \Vect_G(\hX)& \ra &  CH\left({}_KS(G,X)\right)_{\QQ,v}\\ 
\mathcal{E} &\mapsto & ch([\mathcal{E}]_K)\end{array}\right.
\end{equation}
where $_v$ indicates the eigenvectors for the volume character. 
If $\mathcal{E} \in \Vect_G(\hX)$  is in the image of  $\Rep(G)$ via the factorization
 $$\Rep (G)\to \Rep(K_h) \ra \Vect_G(\hX),$$
 then $[\mathcal{E}]$ is endowed with a natural flat connection, so its higher Chern classes in $\oplus_i H^{2i}({}_KS(G,X),\QQ(i))$ are  equal to $0$.  We obtain a morphism 
\begin{equation}
\label{02}
K_0(\Rep(K_h))\otimes_{K_0(\Rep(G))} \QQ  \to  \oplus_i H^{2i}\left({}_KS(G,X),\QQ(i)\right)_v \subset \oplus_i H^{2i}\left({}_KS(G,X),\QQ(i)\right),
\end{equation}
analogous to \eqref{01},
where the subscript $_v$ denotes a certain eigenspace for the action of the unramified Hecke algebra.

  \section{Chern classes for compactified Shimura varieties}
 
 \subsection{Toroidal and minimal compactifications, and canonical extensions of automorphic vector bundles} 
  Henceforward we assume  the Shimura variety ${}_KS(G,X)$ is  {\it not projective}; equivalently, the derived subgroup $G^{{\rm der}}$ of $G$ is isotropic over $\QQ$.  In this case the automorphic theory naturally gives information about Chern classes of canonical extensions on toroidal compactifications on the one hand; on the other hand, the $v$-eigenspace most naturally appears in intersection cohomology of the minimal compactification.   This has been worked out in detail  in the $\mathcal{C}^\infty$ and the $L_2$ theory by Goresky and Pardon in \cite{GP02}.  In what follows, we let    $ j_K: {}_KS(G,X)  \hookrightarrow {}_KS(G,X)^{\rm min}$ denote the minimal (Baily-Borel-Satake) compactification.  The minimal compactification is functorial, thus if $K' \subset K,$  there is  a unique morphism $_{K'}S(G,X)^{\rm min} \ra {}_KS(G,X)^{\rm min}$  extending the natural map of open Shimura varieties, and if $g \in G(\af),$ then the Hecke correspondence $T(g)$ extends canonically to a correspondence on ${}_KS(G,X)^{\rm min} \times {}_KS(G,X)^{\rm min}$   where the product is taken over the reflex field.  In particular, for any cohomology theory $H$ as above, we have Hecke operators 
  $$T(g) \in {\rm End}\big(H^*({}_KS(G,X)^{\rm min})\big).$$

  The minimal compactification is always singular, except when $G^{ad}$ is a product of copies of $\PGL(2)$,  and the automorphic vector bundles do not in general extend as bundles to  ${}_KS(G,X)^{\rm min}$.  Thus the classical theory does not automatically  attach Chern classes to automorphic vector bundles on non-proper Shimura varieties in some cohomology theory
  $H^{*}_?({}_KS(G,X)^{\rm min},*)$.

  On the other hand, there is a large collection of {\it toroidal compactifications}  ${}{}_KS(G,X) \hookrightarrow  {}{}_KS(G,X)_{\Sigma}$  indexed by combinatorial data $\Sigma$  (see \cite[III, \S6, Main Theorem]{AMRT75} for details; the adelic construction is in \cite{Har89, Pin90}). The set of $\Sigma$ is adapted to the level subgroup $K$.  It is partially ordered by {\it refinement}:  if $\Sigma'$ is a refinement of $\Sigma$, then there is a natural proper morphism 
  $$p_{\Sigma',\Sigma}: {}{}_KS(G,X)_{\Sigma'} \ra {}{}_KS(G,X)_{\Sigma}$$
  extending the identity map on the open Shimura variety.  Any two $\Sigma$ and $\Sigma'$ can be simultaneously refined by some $\Sigma''$.   Further, the open embedding $ 
   {}{}_KS(G,X)  \hookrightarrow {}{}_KS(G,X)_{\Sigma} ,$ is completed to a diagram
  \ga{2.1}{
  \vcenter{\vbox{\xymatrix{
  &  {}{}_KS(G,X)_{\Sigma} \ar[d]^{\varphi_\Sigma}\\ 
\ar@{^{(}->}[ru]^{j_{K, \Sigma}}   {}{}_KS(G,X) \ar@{^{(}->}[r]^{j_K}&  {}{}_KS(G,X)^{\rm min}
  }}
  }}
  where $\varphi_\Sigma$ is a proper morphism, which is an isomorphism on  ${}{}_KS(G,X) $.  Now  assume that $K$ is a {\it neat} open compact subgroup of $G(\af)$, in the sense of  \cite{Har89} or \cite[0.6]{Pin90}.  In that case, we can choose $\Sigma$ so that  ${}_KS(G,X)_\Sigma$ is smooth and projective, and the boundary divisor has normal crossings.
{\it We do so unless we specify otherwise.}  With these conventions, $\varphi_\Sigma$ is always a desingularization of the minimal compactification, which is constructed as in \cite{AMRT75}, {\it loc.cit.}; moreover, for any refinement $\Sigma'$ 
 of $\Sigma$, $ {}_KS(G,X)_{\Sigma'}$ is again smooth and projective.    
  
Mumford proved in \cite[Theorem~3.1]{Mum77} that, if $\CE \in {\rm Vect}^{\rm ss}_G(\hX)$ (recall  from \cite{EH17} that the upper index $^{\rm ss}$ stands for semi-simple) the automorphic vector bundle $[\CE]{}_K$ on ${}_KS(G,X)$ admits a {\it canonical extension $[\CE]{}_K^{\rm can}$} to ${}_KS(G,X)_{\Sigma}$; we write  $[\CE]{}_K^{\Sigma}$ if we want to emphasize the toroidal data.  The adelic construction is carried out in Section 4 of \cite{Har89}, where Mumford's result was generalized to arbitrary $\CE \in {\rm Vect}_G(\hX)$.  In particular it was shown there that if $\CE = (G \times W)/P_h$, 
 where $W$ is the restriction to $P_h$ of a representation of $G$, the action of $P_h$ is diagonal, and we have identified $G/P_h$ with $\hX$, then $[\CE]{}_K$ is a vector bundle on ${}_KS(G,X)$ with a flat connection, and its canonical extension $[\CE]{}_K^{\rm can}$ on ${}_KS(G,X)_{\Sigma}$ is exactly Deligne's canonical extension.  In particular,  the connection has logarithmic poles along ${}_KS(G,X)_{\Sigma} \setminus {}_KS(G,X)$.
 Moreover, if $\Sigma'$ is a refinement of $\Sigma$, then 
 $$p^*_{\Sigma',\Sigma}[\CE]{}_K^{\Sigma} = [\CE]{}_K^{\Sigma'}$$
 (see \cite[Lemma~4.2.4]{Har90}),
which  in particular implies 
 $$p_{\Sigma',\Sigma,*}[\CE]{}_K^{\Sigma'} = Rp_{\Sigma',\Sigma,*}[\CE]{}_K^{\Sigma'} = [\CE]{}_K^{\Sigma}.$$  
 Finally, for any fixed $\Sigma$, 
 $$\CE \ra [\CE]{}_K^{\rm can}$$ 
  is a monoidal functor from ${\rm Vect}_G(\hX)$ to the category of vector bundles on ${}_KS(G,X)_\Sigma$.    
  
  Unfortunately, 
the $g \in G(\af)$ generally permute the set of $\Sigma$ and thus the Hecke correspondences $T(g)$ in general {\it do not} extend to correspondences on  a given ${}_KS(G,X)_\Sigma$.   Thus the arguments of \cite{EH17}, which are based on the study of the $v$-eigenspace for the Hecke operators in the cohomology of ${}_KS(G,X)$, cannot be applied directly to prove vanishing of higher Chern classes of $ [\CE]{}_K^{\rm can}$ 
in continuous $\ell$-adic cohomology of ${}_KS(G,X)_\Sigma$  over its reflex field, 
when $[\CE]_K$ is a flat automorphic vector bundle.

One can obtain information about the classes of $[\CE]_K$ on the open Shimura variety, but these lose  information.  From the standpoint of automorphic forms, the natural target of the topological Chern classes of automorphic vector bundles should be the {\it intersection cohomology} of the minimal compactification. 
We fix cohomological notation: $H^*(Z, \mathbb{Q})$, resp. $IH^*(Z, \mathbb{Q})$   denotes Betti resp. intersection cohomology of a complex variety $Z$, while $H^*(Z, \mathbb{Q}_\ell)$, resp. $IH^*(Z, \mathbb{Q}_{\ell})$ denote $\mathbb{Q}_\ell$-\'etale cohomology.

  \begin{rmk}  \label{rmk:hecke} {\rm The action of Hecke correspondences on Betti intersection cohomology $IH({}_KS(G,X), \mathbb{Q})$  are defined analytically by reference to Zucker's conjecture.  For a purely geometric construction of the action of Hecke correspondences on $IH({}_KS(G,X), \mathbb{Q})$ and thus on 
    $IH({}_KS(G,X), \mathbb{Q}_\ell)$ see \cite[(13.3)]{GM03} (the argument applies more generally to {\it weighted cohomology} as defined there).}
  \end{rmk}

The following statement generalizes Proposition~1.20 of \cite{EH17} and is proved in the same way.  
 \begin{prop}\label{isomQ}   There is a canonical isomorphism of algebras
 \ga{}{H^*(\hX,\QQ) \isoarrow IH^*({}_KS(G,X)^{\rm min},\QQ)_v \notag\\ 
 H^*(\hX,\QQ_\ell) \isoarrow IH^*({}_KS(G,X)^{\rm min},\QQ_\ell)_v.\notag
 }
 \end{prop}

 \begin{proof}  
 The second statement is deduced directly from the first one by the comparison isomorphism  \cite[Section~6]{BBD82}.
 As in the proof of the analogous fact in \cite{EH17}, it suffices to prove the corresponding statement over $\CC$.   By Zucker's Conjecture \cite{Loo88, SS90},   $IH^*({}_KS(G,X)^{\rm min}, \CC)$ is computed using {\it square-integrable} automorphic forms and Matsushima's  formula \cite[3.6, formula (1)]{Bor83}.  Say
 the space $\CA_{(2)}(G)$ of square-integrable automorphic forms on $G(\QQ)\backslash G(\ad)$ decomposes as the direct sum
 $$\CA_{(2)}(G) = \oplus_{\pi} m(\pi)\pi$$
 where $\pi$ runs over irreducible admissible representations of $G(\ad)$ and $m(\pi)$ is a non-negative integer, which is positive for a countable set of $\pi$ \footnote{Automorphic forms are understood to be $C^\infty$ and $K_h$-finite, and to satisfy the remaining conditions introduced by Harish-Chandra, hence no completion is needed.}.  Then
 $$IH^i({}_KS(G,X)^{\rm min},\CC) \isoarrow \oplus_{\pi} m(\pi)(H^i(\fg,K_h; \pi_{\infty})\otimes \pi_f^{K}).$$
 This implies
\begin{equation}\label{matsu}
IH^i({}_KS(G,X),\CC)_v \isoarrow \oplus_{\pi} m(\pi)(H^i(\fg,K_h; \pi_{\infty})\otimes (\pi_f^{K})_v),
\end{equation}
 where $(\pi_f^{K})_v$ is the eigenspace in $\pi_f^K$ for the volume character  of $\CH{}_K$.  Write $\pi_f = \otimes'_q \pi_q$, where $q$ runs over rational primes.  Now if $q$ is unramified for $K$ then $\pi_f^K = 0$ unless $\pi_q$ is spherical.   Now the trivial representation of $G(\QQ_q)$ has the property that it equals its spherical subspace and the corresponding representation of the local Hecke algebra is the (local) volume character.  By the Satake parametrization the trivial representation is the only spherical representation of $G(\QQ_q)$ with this property.  Thus   $(\pi_f^{K})_v \neq 0$ implies that $\pi_q$ is the trivial representation for all $q$ that are unramified for $K$.  It then follows from weak approximation that $\pi$ is in fact the trivial representation.  Thus for all $i$,
 \begin{equation}\label{HHv} IH^i({}_KS(G,X),\CC)_v \isoarrow H^i(\fg,K_h; \CC).\end{equation}
 But this is equal to $H^i(\hX, \CC)=IH^i(\hX,\CC)$ by a standard calculation; see \cite[Remark~16.6]{GP02}.  The rest of the proof follows as in the proof of  \cite[Proposition~1.20]{EH17}.
  \end{proof}

The point of the preceding proposition is that the Chern classes of automorphic vector bundles on the {\it non-proper} Shimura variety ${}_KS(G,X)$ are represented by square-integrable automorphic forms,  more precisely, by the differential forms on $G(\QQ)\backslash G(\ad)$ that are invariant under the action of  $G(\ad)$. In other words, the only $\pi$ that contributes is the space of constant functions on $G(\QQ)\backslash G(\ad)$, which are square integrable modulo the center of $G(\RR)$.  Thus $IH^{2*}({}_KS(G,X)^{\rm min},\mathbb{Q}(*))_v$ can be viewed as $L_2$-Chern classes of automorphic vector bundles. In addition 
 $IH^{2*+1}({}_KS(G,X)^{\rm min},\mathbb{Q}(*))_v=0.$
Although most automorphic vector bundles do not extend  as bundles to the minimal compactification ${}_KS(G,X)^{\rm min}$, it was proved by Goresky and Pardon in \cite[Main Theorem]{GP02}  that
under the natural homomorphism
\begin{equation}
\label{HIH} H^*({}_KS(G,X)^{\rm min},\CC)_v \ra IH^*({}_KS(G,X)^{\rm min},\CC).
\end{equation}
 the classes in 
$IH^*({}_KS(G,X)^{\rm min},\CC)_v$ lift canonically, as differential forms, to  ordinary singular cohomology $H^*({}_KS(G,X)^{\rm min},\CC)_v$.

\medskip

 We use the notation of the diagram~\eqref{2.1}. 
 Let $K$ be a neat open compact subgroup of $G(\af)$, as above.  Let $[\CE]_K$ be the automorphic vector bundle on ${}_KS(G,X)$ attached to the homogeneous vector bundle $\CE \in {\rm Vect}_G(\hX).$  Let $$c_n([\CE]^{\rm can}_K) \in H^{2n}({}_KS(G,X)_\Sigma,\mathbb{Q}(n)), \  
 c_n([\CE]_K) \in H^{2n}({}_KS(G,X),\mathbb{Q}(n))$$ 
 be the Chern classes in Betti cohomology. 
 The following  theorem summarizes the main results of the article \cite{GP02}.

 \begin{thm}\label{gorpar}   \begin{itemize}
 \item[\rm (i)]   There are canonical  classes $c_n([\CE]_K)^{GP} \in
  H^{2n}({}_KS(G,X)^{\rm min},\CC)$ such that
$$\varphi_\Sigma^*(c_n([\CE]_K)^{GP})=c_n([\CE]^{\rm can}_K) \in 
  H^{2n}({}_KS(G,X)_{\Sigma},\QQ(n)).$$
   \item[\rm (ii)]  In particular
 \ga{}{j_K^*(c_n([\CE]_K)^{GP}) = c_n([\CE]_K)  \in  H^{2n}({}_KS(G,X),\QQ(n)). \notag}
   \item[\rm (iii)]  The classes $c_n([\CE])^{GP}$ (we drop the subscript $K$) are represented by square integrable automorphic forms.  More precisely, the image of $c_n([\CE])^{GP}$ under the morphism of \eqref{HIH} is contained in the subspace of intersection cohomology that corresponds, under the isomorphism \eqref{matsu}, to the relative Lie algebra cohomology of $\oplus \pi$ with   trivial archimedean component $\pi_{\infty}$.
  \item[\rm (iv)] 
 Suppose $G^{\rm der}(\RR) = \prod_i G_i$ is a product of simple Lie groups, all of which are of the form $\SP(2n,\RR)$, $SU(p,q)$, $\SO^*(2n)$,   or $\SO(p,2)$ with $p$ odd \footnote{Some of the $G_i$ can also be of the form $\SO(2,2)$, as in the statement of \cite{GP02}.}. 
 The $\mathbb{Q}$-subalgebra   $$H^{2*}_{{\rm Chern}}({}_KS(G,X)^{\rm min}) \subset H^{2*}({}_KS(G,X)^{\rm min},\CC)$$    generated by the $c_n([\CE])^{GP}$, as $\CE$ varies over ${\rm Vect}_G(\hX),$   is endowed with a naturally defined surjective homomorphism
 $$H^{2*}_{\rm Chern}({}_KS(G,X)^{\rm min}) \overset{h}\to H^{2*}(\hX,\QQ(*))$$
 such that for any $n$ the  diagram 
\ga{}{\CD      H^{2n}_{{\rm Chern}}({}_KS(G,X)^{\rm min}) @>>{\rm natural \ map}> IH^{2n}({}_KS(G,X)^{\rm min},\QQ(n))_v \\
                        @VhVV                           @A\isoarrow A{\rm Prop}.\ref{isomQ}A \\
                        H^{2n}(\hX,\QQ(n)) @>=>>   H^{2n}(\hX,\QQ(n))
                        \endCD \notag}
commutes. 
 In other words, the isomorphism of Proposition~\ref{isomQ} tensor $\CC$  factors through 
 \ga{}{H^{2n}({}_KS(G,X)^{\rm min}, \CC) \xrightarrow{{\rm natural \ map}} IH^{2n}({}_KS(G,X)^{\rm min},\CC). \notag}\

 \end{itemize}
 
   \end{thm}

  In the next sections, we  use Peter Scholze's perfectoid geometry and  his Hodge-Tate morphism to prove an analogue of the Goresky-Pardon theorem for $\ell$-adic cohomology of Shimura varieties of abelian type.    
  
  \subsubsection{Even-dimensional quadrics}\label{evenquad}  Theorem \ref{gorpar} excludes the case where $G^{\rm der}(\RR)$ contains a factor isomorphic to $\SO(2k-2,2)$ with $k > 2$.  
  Assuming $G^{\rm der}$ is $\QQ$-simple, there is then a totally real field $F$, with $[F:\QQ] = d$, say, such that $\hat{X}_{\CC}$ is isomorphic to a product $Q_n^d$ of $d$ smooth projective complex quadrics $Q_n$, each of dimension $n = 2k-2$.  
  The reason for this exclusion is explained in \S 16.6 of \cite{GP02}.  Following \S 16.5 of \cite{BH58}, we consider the cohomology algebra $A := H^*(Q_n,C)$ of a complex quadric $Q_n = \SO(n+2)/\SO(n)\times \SO(2)$ of dimension $n$, with $C = \CC$ or $\QQ_\ell$.  
  Then $A$ contains a subalgebra $A^+$ isomorphic to $C[c_1]/(c_1^{n+1})$, with $c_1 \in H^2(Q_n,C)$ given by the Chern class of the line bundle corresponding to the standard representation of $\SO(2)$.  (There is a misprint in \cite{GP02}; the total dimension of $A^+$ as $C$-vector space is $n+1$, not $n$.)  
  Moreover there is an isomorphism
  \begin{equation}\label{euler}    A = A^+[e]/(e^2 - c_1^n) \end{equation}
 with $e$ the Euler class of the the vector bundle arising from the standard representation of $\SO(n)$.  In \cite{BH58}, the class $c_1$ is denoted $x_1$ and the class $e$ is denoted $\prod_{i = 2}^n x_i$; the equation $e^2 = c_1^n$ then follows immediately for formula (6) of \cite{BH58}.   
 
More generally, if $\hat{X}_{\CC}$ is isomorphic to $ Q_n^d$ as above, we denote by $c_{1,r} \in 
H^{2}(\hat{X}_{\CC},C(1))$ the class of the line bundle defined above corresponding to the $r$-th factor of $Q_n^d$, $r = 1, \dots, d$, and let $e_r$ correspond to the Euler class in the $r$-th factor.
 The isomorphism of Proposition~\ref{isomQ} is valid in all cases, and Goresky and Pardon showed that the image in $IH^{2j}({}_KS(G,X)^{\rm min},\QQ(j))_v$ of the classes $c_{1,r}^j \in H^{2j}(\hat{X}_{\CC},C(j))$ lift canonically to the cohomology of the minimal compactification (the twist $j$ here is unnecessary for $C=\mathbb{C}$, 
 we write it for the case   $C=\mathbb{Q}_\ell$).
 However, when $n = 2k-2 > 2$, they were unable to show that the $e_r$ lift.  We can extend the statement of Theorem \ref{gorpar} with the following definition.
 \begin{defn}\label{plusdef}  {\rm (i)  Suppose $G^{\rm der}$ is $\QQ$-simple and $G^{\rm der}(\RR) \simeq \SO(2k-2,2)^d$ for some integer $d$, with $k > 2$.  Suppose $C = \CC, \QQ_\ell$, or $\QQ$.  Define
  \begin{equation}\label{plus}
  \begin{aligned}  H^*(\hat{X},C)^+ &\subset H^*(\hat{X},C); \\  IH^{*}({}_KS(G,X)^{\rm min},\QQ(n))^+_v &\subset  IH^{*}({}_KS(G,X)^{\rm min},\QQ(n))_v
  \end{aligned}
 \end{equation}
 to be the subalgebras generated respectively by the classes $c_{1,r}$, $r = 1, \dots, d$, and their images under the isomorphism of Proposition~\ref{isomQ}.  We similarly define
 \begin{equation}{\label{Kplus}} [K_0(\Rep(K_h))\otimes_{K_0(\Rep(G))} C]^+ \subset K_0(\Rep(K_h))\otimes_{K_0(\Rep(G))} C \end{equation}
 to be the inverse image of $H^*(\hat{X},C)^+$ under the isomorphism \eqref{01}.  
 
 (ii) If $G$ satisfies condition (iv) of Theorem \ref{gorpar}, we let 
 $$H^*(\hat{X},C)^+ = H^*(\hat{X},C);$$  $$IH^{*}({}_KS(G,X)^{\rm min},\QQ(n))^+_v =  IH^{*}({}_KS(G,X)^{\rm min},\QQ(n))_v;$$ 
$$ [K_0(\Rep(K_h))\otimes_{K_0(\Rep(G))} C]^+ = K_0(\Rep(K_h))\otimes_{K_0(\Rep(G))} C.$$

(iii)  In either case, we define 
$${\rm Vect}_G(\hat{X})^+ = ch_{\hat{X}}^{-1}(H^*(\hat{X},C)^+) \subset K_0({\rm Vect}_G(\hat{X})).$$
We continue to use the notation $\mathcal{E}$ to denote (virtual) bundles in $Vect_G(\hat{X})^+$.}
 \end{defn} 
 
\begin{thm}\label{gorpar2}  Suppose $G$ satisfies either {\rm (i)} or {\rm (ii)} of Definition~\ref{plusdef}.  Then the conclusions of Theorem~\ref{gorpar} hold with $H^*(\hat{X},C)$, ${\rm Vect}_G(\hat{X})$, and
$IH^{*}({}_KS(G,X)^{\rm min},\QQ(n))_v$  replaced by the versions with superscript $^+$, and for $\mathcal{E} \in Vect_G(\hat{X})^+$.
\end{thm}

The theorem, and its application in Proposition \ref{Chern} below, naturally extend to groups $G$ such that $G^{\rm der}$ is a product of groups of type (i) and (ii) in Definition \ref{plusdef}.  We omit the details.

\subsection{Perfectoid Shimura varieties and the Hodge-Tate morphism} \label{s:HT}
The results of the present section are entirely due to Scholze \cite{Sch15}, then to Caraiani-Scholze \cite{CS17}
and Pilloni-Stroh \cite{PS16}.

\medskip

 Fix a level subgroup  $K \subset G(\af)$.  Let $G(\af)\to G(\Qp), \ k\mapsto k_p$ be the projection.  
Denote by $K_p$ its projection to  $G(\Qp)$. 
  For $r \geq 0$ we let   $K_{p,r} \subset K_p$  be a decreasing family of subgroups of finite index, with $K_{p,r} \supset K_{p,r+1}$ for all $r$, and such that $\bigcap_r  K_{p,r} = \{1\}$.  Let $K_r = \{k \in K ~|~ k_p \in K_{p,r}\}$, and let $K^p = \cap_r K_r$.   We  identify $K^p$ with its projection to the prime-to-$p$ ad\`eles $G(\mathbf{A}_f^p)$; then $K^p$ is an open compact subgroup of $G(\mathbf{A}^p_f)$  called a ``tame level subgroup".  

We assume that the Shimura datum $(G,X)$ is of {\it Hodge type}.  Thus, up to replacing $K$ by a subgroup of finite index, ${}_KS(G,X)$ admits an embedding of Shimura varieties in a Siegel modular variety of some level attached to the Shimura datum $(\GSp(2g),X_{2g})$ for some $g$, where $X_{2g}$ is the union of the Siegel upper and lower half-spaces.  We let $\hat{X}_{2g}$ denote the compact dual flag variety of $X_{2g}$.  Let $C$ denote the completion of an algebraic closure $\Qpb$ of $\QQ_p$.    Denote by ${}_{K_r}\CS(G,X)$, resp. ${}_{K_r}\CS(G,X)^{\rm min}$
 the adic space over $C$ attached to ${}_{K_r}S(G,X)$, resp. ${}_{K_r}S(G,X)^{\rm min}$.   We assume that $K^p$ is contained in the principal congruence subgroup of level $N$ for some $N \geq 3$, in the sense explained in \S 4 of \cite{Sch15}.  
  We restate one of the main theorems of Scholze's article.    
 In \S 4.1 of \emph{loc. cit.} 
it is shown that there is a level $K'_r$ for $(\GSp(2g),X_{2g})$ such that $K_r=K'_r\cap G(\af)$, and such that the scheme theoretic image of ${}_{K_r}S(G,X)^{\rm min}$ in 
${}_{K'_r}S(\GSp(2g),X_{2g})^{\rm min}$ does not depend on the choice of $K'_r$. We denote it by 
$\overline{{}_{K_r}S(G,X)^{\rm min}}$, and by  $\overline{{}_{K_r}\CS(G,X)^{\rm min}}$ its associated $C$-adic space. 

 \begin{thm}[Theorem 4.1.1 in \cite{Sch15}]  \label{perfec2}
There is a  commutative diagram of morphisms of adic spaces \footnote{The projective limits on the right-hand side are not literally adic spaces; see Discussion \ref{limitdiscussion} for the precise meaning of this statement.}
\ga{perfec}{
\vcenter{\vbox{
\xymatrix{ \ar@{^{(}->}[d]_{j}  {}_{K^p}\CS(G,X) \ \ \ \ \  \ \ \ \sim & \ar@{^{(}->}[d]^{ \varprojlim_r j_r} \varprojlim_r   {}_{K_r}\CS(G,X)\\
 \overline{{}_{K^p}\CS(G,X)}^{\rm min}  \ \ \ \ \  \sim  &  \varprojlim_r \overline{{}_{K_r}\CS(G,X)}^{\rm min}.
}}
}}

Here, on the left side, $j$ is an open embedding of  perfectoid Shimura varieties over $C$,  on the right side, $\varprojlim_r j_r$ denotes the formal inverse system of open embeddings of  adic spaces over $C$, and the notation $\sim$  is defined precisely in  \cite[Definition~2.4.1]{SW13}.  
Moreover, there is a $G(\Qp)$-equivariant  Hodge-Tate morphism 
$$\pi_{HT}:  \overline{{}_{K^p}\CS(G,X)}^{\rm min} \ra {\hat{\mathcal{X}}}_{2g}$$
which is compatible with change of tame level subgroup $K^p$.
\end{thm}

\begin{disc}\label{limitdiscussion} {\rm
Here we use the notation $\hat{\mathcal{X}}_{2g}$  for  the adic space over $C$ attached to the flag variety $X_{2g}$.  In {\it loc. cit.}  the notation $\mathcal{X}^*_{K^p}$  is used 
for the perfectoid Shimura variety  ${}_{K^p}\CS(G,X)^{\rm min}.$    The notation $\sim$ indicates that the right hand side of \eqref{perfec} is not to be viewed as a projective limit in the category of adic spaces, which in general does not make sense.   Rather, what is meant is that, for each $r$, there is a commutative diagram
\ga{perfec-r}{
\vcenter{\vbox{
\xymatrix{ \ar@{^{(}->}[d]_{j}  {}_{K^p}\CS(G,X) \   \ar[r]  \   & \ar@{^{(}->}[d]^{j_r}   {}_{K_r}\CS(G,X)\\
 \overline{{}_{K^p}\CS(G,X)}^{\rm min}  \ \  \ar[r] \    &\overline{{}_{K_r}\CS(G,X)}^{\rm min};
}}
}}
that these are compatible with the natural maps from level $K_{r+1}$ to level $K_r$; and that the objects on the left have the properties indicated in  \cite[Definition~2.4.1]{SW13}.}
\end{disc}

The target of the Hodge-Tate morphism was clarified in  \cite{CS17}.
Let $\hat{\mathcal{X}}\subset \hat{\mathcal{X}}_{2g}$  be the closed embedding of $C$-adic spaces corresponding to the closed embedding  $\hat{X}\subset \hat{X}_{2g}$ of the compact duals of $X $  and $X_{2g}$. 

\begin{thm}[Theorem~2.1.3 in \cite{CS17}]\label{thmCS}  The Hodge-Tate morphism $\pi_{HT}$ in Theorem~\ref{perfec2} factors through  the inclusion
$\hat{\mathcal{X}} \subset \hat{\mathcal{X}}_{2g}$, yielding
  a $G(\mathbb{Q}_p)$-
 equivariant Hodge-Tate morphism 
$$\pi_{HT}:  ~~\overline{_{K^p}\CS(G,X)}^{\rm min} \ra \hat{\mathcal{X}}.$$

\end{thm}
\begin{proof}  The existence of $\pi_{HT}$  is stated for the {\it open} perfectoid Shimura variety $_{K^p}\CS(G,X)$ with values in $\hat{\mathcal{X}}$.  Since $\hat{\mathcal{X}}$ is closed in $\hat{\mathcal{X}}_{2g}$, the extension to the boundary then follows from Theorem~\ref{perfec2}
by continuity,  as in \cite[Theorem 3.3.4]{C+6}.
\end{proof}

Assume for the moment that $(G,X) = (\GSp(2g),X_{2g})$.   Fix $K_g \subset \GSp(2g,\af)$ a neat compact open subgroup, with $K_p = \GSp(2g,\Zp)$ and write $_{K_g}\CA_g$ instead of ${}_{K_g}S(\GSp(2g),X_{2g})$ for the Siegel modular variety of genus $g$ and level $K_g$, viewed as an adic space over $C$.   Let $K_{g,r} \subset K_g$ be the principal congruence subgroup of $K_g$ of level $p^r$.
Let $_{K_g}\CA_g^{\rm tor} = {}_{K_g}\CA_{g,\Sigma_g}$ be a smooth projective toroidal compactification of ${}_{K_g}\CA_g$ for some combinatorial datum $\Sigma_g$ as above.  Following Pilloni-Stroh in  \cite{PS16}, \S 1.3, but with a change of notation, we let ${}_{K_{g,r}}\CA_g^{\rm tor}$   denote the corresponding toroidal compactification of ${}_{K_{g,r}}\CA_g$ for each $r$, with the same $\Sigma_g$.  The ${}_{K_{g,r}}\CA_g^{\rm tor}$ form a projective system of adic spaces over $C$.  The authors construct a projective system of normal models ${}_{K_{g,r}}\mathcal{A}_{g, \mathcal{O}_C}^{\rm tor}$ over ${\rm Spec}(\mathcal{O}_C)$ for each $r$, and define a perfectoid space $_{K^p_g}\mathcal{A}_g^{\rm tor}$ which is the generic fiber of the projective limit
of the  ${}_{K_{g,r}}\mathcal{A}_{g, \mathcal{O}_C}^{\rm tor}$, in the  sense of
\cite[Section~2.2]{SW13}.  See \cite[Section~A.12 and Corollaire~A.19]{PS16} for the statement that the projective limit, which they denote $\mathcal{X}(p^{\infty})^{\rm tor-mod}$,  is indeed perfectoid.    There is a natural map
$$\mathfrak{q}_g:  _{K^p_g}\mathcal{A}_g^{\rm tor} \ra _{K^p_g}\mathcal{A}_g^{\rm min}.$$ 

Let $(G,X)$ be any Shimura datum of Hodge type, with $\iota:  (G,X) \hookrightarrow (\GSp(2g),X_{2g})$ a fixed symplectic embedding.  In the remainder of this section we make the following hypothesis, which (as we have seen) is true for the Shimura datum $(\GSp(2g),X_{2g})$:
\begin{hyp}\label{projtor} The projective system of ${}_{K_r}S(G,X)$ has a compatible projective system of toroidal compactifications
${}_{K_r}S(G,X)_\Sigma$ such that, if ${}_{K_r}\CS(G,X)_\Sigma$ is the associated adic space over $C$, then there is a perfectoid space ${}_{K^p}\CS(G,X)_\Sigma$ such that
$${}_{K^p}\CS(G,X)_\Sigma ~\sim ~ \varprojlim {}_{K_r}\CS(G,X)_\Sigma$$
in the sense of \cite[Section~2.2]{SW13}.  
\end{hyp}

It is likely the  Hypothesis~\ref{projtor} is valid, but its proof is probably as elaborate as the construction of integral toroidal compactifications in \cite{Lan13} \footnote{Ana Caraiani has informed us that a proof of Hypothesis \ref{projtor}, for the case of a Shimura variety attached to a quasi-split unitary similitude group of even dimension, will appear in her forthcoming joint work with Scholze.}. Thus we will provide an alternative proof  of our main theorems, that does not depend on this hypothesis, in section \S \ref{sec:covers}.  However, the proof assuming Hypothesis \ref{projtor} is considerably simpler, and it is valid, by \cite{PS16}, in the Siegel modular case.  

Under Hypothesis \ref{projtor}, there is automatically a map 
$$q:  {}_{K^p}\CS(G,X)_\Sigma \ra {}_{K^p}\CS(G,X)^{\rm min}$$
as well as the maps
 $$q_r:  {}_{K_r}\CS(G,X)_\Sigma \ra {}_{K_r}\CS(G,X)^{\rm min}$$
 at finite level (the existence of which does not depend on the hypothesis).

By  \cite[Section~3.4]{Har85},  there is a correspondence 
\ga{bb}{
\vcenter{\vbox{
\xymatrix{ {}_KI(G,X) \ar[d]_{b'} \ar[r]_{a'} & \hat{X} \\
 {}_KS(G,X)
} }
}}
where $b'$ is 
a family of $G$-torsors,  functorial with respect to inclusions $K' \subset K$ and translation by elements $g \in G(\ad_f)$, and $a'$ is a $G$-equivariant morphism.  
For any $G$-equivariant vector bundle $\mathcal{E}$ over $\hat{X}$, the pullback $a^{\prime*}(\mathcal{E})$ is $G$-equivariant, and thus descends to the automorphic vector bundle $[\mathcal{E}]$ over ${}_KS(G,X)$.
The $G-\rm{torsor}$ ${}_KI(G,X)$ is constructed as the moduli space $\Isom^{\otimes}([V_{\rho}], V)$ of trivializations of a flat automorphic vector bundle $[V_{\rho}]$ attached to a faithful representation $\rho:  G \ra \GL(V)$; the superscript $^{\otimes}$ indicates that the isomorphisms respect absolute Hodge cycles.  
Thus the construction requires {\it a priori} the existence of the automorphic vector bundle $[V_{\rho}]$.  
When $ {}_KS(G,X)$ is of Hodge type, one can take $\rho$ to be a symplectic embedding $\rho:  G \ra \GSp(V)$, and define $[V_{\rho}]$ to be the pullback to $ {}_KS(G,X)$ of the (dual of the) relative de Rham $H^1$ of the universal abelian scheme over the Siegel modular variety attached to $\GSp(V)$.   
Then the morphism $a$ is defined by transferring the Hodge filtration on $[V_{\rho}]$ to the constant vector bundle defined by $V$.   That this construction is canonically independent of the choice of symplectic embedding $\rho$ is explained in \cite[Remark 4.9.1]{Har85}; see also \cite[Lemma 2.3.4]{CS17}.

Choose a base point $h \in X \subset \hX$ as  in Section \ref{sec:autobdl}; we may as well assume $h$ to be a $CM$ point.  
Recall that  $P_h$ denotes the stabilizer of $h$ and  $K_h$  its Levi quotient (in \cite{CS17} this group is denoted $M_\mu$).  
Any faithful representation $\alpha: K_h \ra \GL(W_h)$ defines by pullback to $P_h$ a $G$-equivariant vector bundle $\mathcal{W}_h$ over $\hat{X}$, and thus an automorphic vector bundle $[\mathcal{W}_h]$ over ${}_KS(G,X)$ that varies functorially in $K$.    
We can then define a family (depending on $K$) of $K_h$-torsors $b:T =  {}_{K}T_h(G,X) \ra {}_KS(G,X)$ with a $K_h$-equivariant morphism $a: {}_{K}T_h(G,X) \ra \hat{X}$ as the moduli space of trivializations of $[\mathcal{W}_h]$ as above.  
More precisely, letting $R_uP_h$ denote the unipotent radical of $P_h$, the natural morphism $\hat{X}_h := G/R_uP_h \ra \hat{X}$ is canonically a $G$-equivariant $K_h$-torsor, whose pullback $a^{\prime -1}\hat{X}_h$ descends to a $K_h$-torsor over ${}_KS(G,X)$, over the reflex field of the CM point $h$, that is naturally identified with $T$ defined above.  
Moreover, the construction is canonically independent of the choice of base point.   
This is also constructed in Section~2.3, especially Lemma~2.3.5, of \cite{CS17}.

 The pullback of $T$ via the ringed space morphism $\hat{\mathcal{X}} \to \hat{X}$  is denoted by $\mathcal{T}$.   
  The $K_h$-torsor $b$ has a Mumford extension $ T^{\rm can}\to {}_KS(G,X)_\Sigma$   as a $G$-torsor.  See \cite{HZ94}, specifically Lemma~4.4.2 and pages 320-321, where the Mumford extension of the $G$-torsor $b'$  in \eqref{bb}  is constructed;  the $K_h$-torsor  is constructed in the same way. 
 One denotes by 
  $\mathcal{M}^{\rm can} = {}_K\mathcal{M}^{\rm can}$ the pullback of $T^{\rm can}$ 
   via the ringed space morphism ${}_K\CS(G,X)_\Sigma \to {}_K S(G,X)_\Sigma.$

We define the $K_h$-torsor 
 \ga{}{\mathcal{M}_p := \pi_{HT}^*\mathcal{T}  \notag} 
 on $_{K^p}\CS(G,X)^{\rm min}$, 
 which coincides by definition with the $K_h$-torsor defined in \cite[Lemma~2.3.8]{CS17}.
We define the $K_h$-torsor 
 \ga{can}{ \mathcal{M}_{dR, \Sigma,K}:= \pi_{K, K^p}^{*}\left( {}_K\mathcal{M}^{\rm can}\right)  }
   on ${}_{K^p}\sS(G,X)_{\Sigma}$.

\begin{prop}[Proposition 2.3.9 in \cite{CS17}]\label{compar}  
For any neat level subgroup $K$, there is a canonical isomorphism
\begin{equation*} j^* \CM_{p}  \isoarrow  j_{\Sigma}^* \CM_{dR,\Ss,K}  \end{equation*}
of $K_h$-torsors over $_{K^p}\CS(G,X)$. 
\end{prop}

Although the article \cite{CS17} is written for compact Shimura varieties, the argument  developed there for this point  is valid for any Shimura variety of Hodge type.   Strictly speaking, as explained in \cite{CS17}, the torsors $\CM_{p}$ and $\CM_{dR,\Ss, K}$ have natural extensions to $G$-torsors by pullback to torsors for opposite parabolics, followed by pushforward to $G$, so the comparison only applies to semisimple automorphic vector bundles.

The following theorem is essentially due to Pilloni and Stroh.  
\begin{thm}\label{compar2}  Assume Hypothesis \ref{projtor}.  Then, for any neat level subgroup $K$, the isomorphism of Proposition \ref{compar} extends to a canonical isomorphism 
\ga{}{  q^*\mathcal{M}_p \isoarrow \mathcal{M}_{dR, \Sigma,K} \notag}
of $K_h$-torsors 
over $_{K^p}\CS(G,X)_{\Sigma}$.  As $K^p$ varies,  
these isomorphisms are equivariant under the action of $G(\mathbf{A}_f^p)$. 
\end{thm}

In \cite{PS16} this is proved for the Siegel modular variety, although it is not stated in this form.   In
 Section~\ref{proofofcomparison}  we explain their result and show how to obtain  Theorem~\ref{compar2} for general Shimura varieties of Hodge type, under Hypothesis \ref{projtor}.   
In the following section, we show how to dispense with Hypothesis \ref{projtor} and prove Theorem ~\ref{compar3} as an alternative to Theorem ~\ref{compar2}.
 In Section~\ref{sec:chern} we construct $\ell$-adic Chern classes using either Theorem \ref{compar2} (when it is available) or Theorem ~\ref{compar3} for general Shimura varieties of Hodge type.

 \section{Construction of perfectoid covers}\label{sec:covers}
 
Let $(G,X)$ be any Shimura datum of Hodge type, with $\iota:  (G,X) \hookrightarrow (\GSp(2g),X_{2g})$ a fixed symplectic embedding.  Choose a combinatorial datum $\Sigma$ for $_KS(G,X)$ that is compatible with the fixed $\Sigma_g$ chosen above, so that there is a morphism
$$[\iota]:  {}_KS(G,X)_\Sigma \to {}_{K_g}\CA_{g,\Sigma_g}$$
which factors through  
\ga{bar}{ \overline{{}_{K}S(G,X)_{\Sigma}} = {}_{K_g}\CA_{g,\Sigma_g} \times_{ {}_{K_g}\CA^{\rm min}_g}
\overline{{}_K \CS(G,X)}^{\rm min}.}
We denote the corresponding $C$-adic spaces by 
${}_K \CS(G,X)_\Sigma$ and $ \overline{{}_{K} \CS(G,X)_{\Sigma}} $.

If we admitted Hypothesis \ref{projtor} then we would be able to replace $ \overline{{}_{K} \CS(G,X)_{\Sigma}}$ by ${}_K \CS(G,X)_\Sigma$ in what follows; however, we will not assume Hypothesis \ref{projtor} in the remainder of this section.  
By \cite[Prop.~6.18]{Sch12} the fibre product
\ga{bar2}{ \overline{ {}_{K^p} \CS(G,X)_\Sigma}:= {}_{K^p_g}\mathcal{A}_g^{\rm tor} \times_{ 
{}_{K_g^p} \mathcal{A}_g^{\rm min}} \overline{{}_{K^p}\CS(G,X)}^{\rm min}
}
exists in the category of perfectoid spaces.  We denote by 
\ga{}{ \bar{q}:  \overline{ {}_{K^p} \CS(G,X)_\Sigma} \to \overline{{}_{K^p}\CS(G,X)}^{\rm min} }
 the projection on the second factor, and by
 $$\bar{q}_r:  \overline{ {}_{K^p} \CS(G,X)_\Sigma} \to \overline{{}_{K_r}\CS(G,X)}^{\rm min}$$
 the composition with the canonical map of $\overline{{}_{K^p}\CS(G,X)}^{\rm min}$ to $\overline{{}_{K_r}\CS(G,X)}^{\rm min}$.
By definition,  in the sense of \cite[Definition~2.4.1]{SW13}, one has 
$$
{}_{K^p_g}\mathcal{A}_g^{\rm tor} \sim  \varprojlim_r {}_{K_{g,r}}\CA_g^{\rm tor}, \ \ 
{}_{K^p_g}\mathcal{A}_g^{\rm min} \sim  \varprojlim_r {}_{K_{g,r}}\CA_g^{\rm min}, \ \
\overline{{}_{K^p}\CS(G,X)}^{\rm min} \sim  \varprojlim_r \overline{{}_{K_r}\CS(G,X)}^{\rm min}.
$$

One deduces
\ga{lim}{     \overline{ {}_{K^p} \CS(G,X)_\Sigma} \sim \varprojlim_r 
  \overline{ {}_{K_r} \CS(G,X)_\Sigma}. }
In particular, by \cite[Corollary~7.1.8]{Sch12}, one has the relation
    \ga{coh}{  H^*\left(  \overline{ {}_{K^p} \CS(G,X)_\Sigma}, \mathbb{Z}/\ell^n\right)= \varinjlim_r 
    H^*\left(  \overline{ {}_{K_r} \CS(G,X)_\Sigma}, \mathbb{Z}/\ell^n\right). }
On the other hand,  Scholze's Hodge-Tate map yields the composite map 
\ga{}{  \overline{ {}_{K^p} \CS(G,X)_\Sigma} \to {}_{K_g^p} \mathcal{A}_g^{\rm min} \to \hat{ \mathcal{X}}_{2g} \notag}
which by definition is the same as the composite map 
\ga{}{  \overline{ {}_{K^p} \CS(G,X)_\Sigma} \to \overline{{}_{K^p} \CS(G,X)}^{\rm min}  \xrightarrow{\pi_{HT}} \hat{\mathcal{X}}    \to \hat{ \mathcal{X}}_{2g}. \notag}
This defines a map
\ga{HT}{  \overline{ {}_{K^p} \CS(G,X)_\Sigma} \to \hat{\mathcal{X}} }
in the category of $C$-adic spaces. It is $G(\mathbb{Q}_p)$-equivariant.

In what follows the idea to use the space denoted $\mathbb{P}^{N, {\rm perf}}$
to define a perfectoid space above 
${}_{K_s}\CS(G,X)_{\Sigma}$ 
 is due to Bhargav Bhatt. 
We fix a natural number $s$. Associated to it we have the normalization morphism
$\nu_s:   {}_{K_s}\CS(G,X)_{\Sigma}\to \overline{  {}_{K_s}\CS(G,X)_{\Sigma}}.$ 
It is a finite morphism, thus it factors as 
$$
\nu_s: {}_{K_s}\CS(G,X)_{\Sigma}\xrightarrow{\eta_s}
 \overline{  {}_{K_s}\CS(G,X)_{\Sigma}}\times_{C} \mathbb{P}^N   \xrightarrow{\rm projection} \overline{  {}_{K_s}\CS(G,X)_{\Sigma}}
$$
 for some natural number $N$,  where $\eta_s$ is a closed embedding. We set 
 $$Y_s={\rm Image}(\eta_s).$$
 By definition $\eta_s$ induces an isomorphism between  ${}_{K_s}\CS(G,X)_{\Sigma}$ and $Y_s$.  For any $\rho \geq 0$ we also define
 \begin{equation}\label{plusrho} Y_{s+\rho}  = Y_s \times_{ \overline{  {}_{K_s}\CS(G,X)_{\Sigma}}}  \overline{ {}_{K_{s+\rho}}\CS(G,X)_{\Sigma}} .
 \end{equation}

We fix coordinates $(x_0:\ldots:x_N)$ of $\mathbb{P}^N$ and define  for $\rho \in \mathbb{N}$ the finite flat morphisms
$\phi_\rho: \mathbb{P}^N\to \mathbb{P}^N$ by $\phi_\rho (x_0:\ldots:x_N)= (x_0^{p^\rho}:\ldots: x_N^{p^\rho})$. 
Over $C$,  a 
 perfectoid space $\mathbb{P}^{N, {\rm perf}}$ is associated to  $(\mathbb{P}^N, \phi_\rho)$
 \cite[Section~17]{Sch14},  such that
$$\mathbb{P}^{N, {\rm perf}}\sim \varprojlim_\rho \phi_\rho.$$
 Then the fiber product
$$ {}_{K^p}\CS(G,X)_{\Sigma,N} := \overline{  {}_{K^p}\CS(G,X)_{\Sigma}}\times_C \mathbb{P}^{N, {\rm perf}}
$$
exists as a perfectoid space, by \cite[Proposition 6.18]{Sch12}, and we have
\begin{equation}\label{sgxN}
  {}_{K^p}\CS(G,X)_{\Sigma,N}  \sim \varprojlim_{\rho}  \overline{  {}_{K_{s+\rho}}\CS(G,X)_{\Sigma}} \times_C \mathbb{P}^N,
\end{equation}
where the map 
$$\overline{  {}_{K_{s+\rho}}\CS(G,X)_{\Sigma}} \times_C \mathbb{P}^N \ra \overline{  {}_{K_s}\CS(G,X)_{\Sigma}} \times_C \mathbb{P}^N$$
is given by the natural projection in the first factor and $\phi_\rho$ in the second factor.

\begin{prop}\label{Zs} There is a closed perfectoid subspace $$Z_s \subset  {}_{K^p}\CS(G,X)_{\Sigma,N}$$ such that
\begin{equation*}
Z_s \sim \varprojlim_{\rho \ge 0}  Y_{s+\rho} .
\end{equation*}
\end{prop}

\begin{proof}
We define $Z_s \subset {}_{K^p}\CS(G,X)_{\Sigma,N}$ by the pullback of the ideal of $Y_s$ in $  \overline{  {}_{K_s}\CS(G,X)_{\Sigma}}\times_{C} \mathbb{P}^N  $  via \eqref{sgxN} 
$$
{}_{K^p}\CS(G,X)_{\Sigma,N}  
 \to   \overline{  {}_{K_s}\CS(G,X)_{\Sigma}}\times_{C} \mathbb{P}^N .  $$ 
It follows from \cite[Proposition~9.4.1]{Bha17} that  $Z_s$ is perfectoid. 
\end{proof}

By construction, $Z_s$ maps to 
$\overline{  {}_{K^p}\CS(G,X)_{\Sigma}} $ which itself maps to 
 $ \overline{{}_{K^p}\CS(G,X)}^{\rm min}$, and to ${}_{K_s}\CS(G,X)_{\Sigma}$, but not necessarily to ${}_{K_r}\CS(G,X)_{\Sigma}$ for $r>s$.
  On the other hand, ${}_{K_s}\CS(G,X)_{\Sigma}$ naturally maps to ${}_{K_s}\CS(G,X)^{\rm min}$, and thus we have a composite diagram
 \begin{equation}\label{composite}
 Z_s \overset{\phi}\to {}_{K_s}\CS(G,X)_\Sigma \overset{q_s}\to  {}_{K_s}\CS(G,X)^{\rm min}   \overset{u_s}\to  \overline{{}_{K_s}\CS(G,X)}^{\rm min} 
\end{equation} 
We let $Q$ denote the composite of the three maps, and let $\underline{q}_s = u_s \circ q_s$ denote the composite of the last two maps.
 
\begin{prop}
\label{CY}  The  perfectoid space $Z_s$ has the following properties.
\begin{itemize}
\item[\rm (i)]  There are morphisms  
$$Q:  Z_s \xrightarrow{\psi} \overline{ {}_{K^p} \CS(G,X)_\Sigma} \xrightarrow{\bar q}    \overline{{}_{K^p} \CS(G,X)}^{\rm min}$$
$$\phi: Z_s \ra {}_{K_s}\CS(G,X)_{\Sigma}$$ 
such that the following diagram commutes 
\ga{}{\xymatrix{\ar[dr]^{Q_s}  \ar[d]_\phi  Z_s \ar[r]^\psi & \overline{ {}_{K^p}\CS(G,X)_\Sigma} \ar[d]^{\bar q_s} \\
Y_s={}_{K_s}\CS(G,X)_\Sigma \ar[r]^{\underline{q}_s} & \overline{{}_{K_s}\CS(G,X)}^{\rm min}
} \notag}

\item[\rm (ii)]  The homomorphism on cohomology
$$\phi^*: H^{\bullet}(Y_s,\mathbb{F}_\ell) \to H^{\bullet}(Z_s, \mathbb{F}_\ell),$$
induced from {\rm (i)}, is injective, where $\ell$ is a prime different from $p$. The injectivity holds also true with $\mathbb{Q}_\ell$-coefficients when $s=0$. 
\end{itemize}
\end{prop}

\begin{proof}  In what follows we assume $G^{ad}$ contains no $\QQ$-simple factor isomorphic to $\PGL(2)$ -- in other words, that $S(G,X)$ contains no factor
isomorphic to the elliptic modular curve -- since in that case the minimal and toroidal compactifications coincide and there is nothing to prove.
The point (i) is just a restatement of the above construction.  
We prove (ii).  By  Proposition~\ref{Zs} together with  \cite[Corollary 7.18]{Sch12}, we have 
$$H^\bullet(Z_s,\mathbb{F}_\ell)= \varinjlim_{\rho \ge 0} H^\bullet(Y_{s+\rho}, \mathbb{F}_\ell).$$
On the other hand, the proper maps 
$Y_{s+\rho}\to Y_{s+\rho-1}$ for $\rho \ge 1$ are finite, of degree a $p$-power for $s\ge 1$. Indeed, the finiteness of this map is reduced by Definition~\eqref{plusrho} to the corresponding assertion for the maps
$$\overline{ {}_{K_{s+\rho}}\CS(G,X)_{\Sigma}} \rightarrow \overline{  {}_{K_s}\CS(G,X)_{\Sigma}}.$$
This in turn is reduced by  \eqref{bar} to the well known fact that  the corresponding morphisms for change of level in the toroidal compactification
${}_{K_g}\CA_{g,\Sigma_g}$ and the compactifications $\overline{{}_K \CS(G,X)}^{\rm min}$ are finite.  
 Indeed, let $\omega_g$ denote the determinant of the relative cotangent bundle of the universal abelian scheme over ${}_{K_g}\CA_g$; in other words, $\omega_g$ is the line bundle whose global sections are Siegel modular forms of weight $1$ and level $K_g$.  Then it is well known (from the theory of the minimal compactification, and  because we have excluded the case of modular curves) that $\omega_g$ extends to an ample line bundle over ${}_{K_g}\CA_g^{\rm min}$, and that the pullback of this extension to ${}_K \CS(G,X)^{\rm min}$ is ample as well.  Thus -- although we have been told there is a problem with the injectivity claim in Corollary 10.2.3 of \cite{SW17} -- the map from ${}_K \CS(G,X)^{\rm min}$ to its image in ${}_{K_g}\CA_g^{\rm min}$ is finite, and
 the morphisms from ${}_K \CS(G,X)^{\rm min}$ to $\overline{{}_K \CS(G,X)}^{\rm min}$ are finite, and  from the commutative square
\ga{}{ \xymatrix{\ar[d]_{\rm finite} {}_K \CS(G,X)_s^{\rm min} \ar[r]^{\rm finite }  & \ar[d] \overline{{}_K \CS(G,X)}_s^{\rm min}\\
{}_K \CS(G,X)_r^{\rm min} \ar[r]^{\rm finite} & \overline{{}_K \CS(G,X)}_r^{\rm min}
} \notag
}
for $s \geq r$ we conclude that the  the level-changing maps 
$\overline{{}_K \CS(G,X)}_s^{\rm min} \to  \overline{{}_K \CS(G,X)}_r^{\rm min}$  are finite.

This implies (ii) for $s\ge 1$ -- and also for $s=0$, after replacing $\mathbb{F}_\ell$ by $\mathbb{Q}_\ell$ in case the order of the group $K_0/K_1$ is divisible by $\ell$. 
\end{proof}

The analogue of Theorem \ref{compar2} is
\begin{thm}\label{compar3}  
For any neat level subgroup $K$ and any $s \geq 0$, the isomorphism of Proposition \ref{compar} extends to a canonical isomorphism  %
\ga{}{ Q^*\mathcal{M}_p \isoarrow \phi^*\mathcal{M}_{dR, \Sigma,K_s} \notag}
of $K_h$-torsors 
over $Z_s$.  As $K^p$ varies,  
these isomorphisms are equivariant under the action of $G(\mathbf{A}_f^p)$. 
\end{thm}

The proof is given in Section \ref{proofofcomparison}.

\section{Construction of $\ell$-adic Chern classes} \label{sec:chern}

As in the previous sections, all level subgroups $K$ are neat and all toroidal compactifications are smooth and projective.  The constructions in the present section depend in an essential way on on Theorem~\ref{compar3} -- or on Theorem \ref{compar2}, if we assume  Hypothesis~\ref{projtor}.   Proofs are given without assuming Hypothesis ~\ref{projtor};  to pass from the situation with Hypothesis~\ref{projtor},  we just formally set $Q=q$. 

\subsection{Construction using $\pi_{HT}$}  \label{ss:piHT}

As in \cite[Theorem~2.1.3\,(2)]{CS17}, 
 the existence of the $K_h$-torsor $\CM_p$ over $\overline{_{K^p}\CS(G,X)}^{\rm min}$ implies that there is a functor
\ga{fp}{   
\begin{array}{rcl}
{\rm Vect}^{\rm ss}_G(\hat{X})[ ~~ \overset{r}\simeq  {\rm Rep}_{K_h}]  & \to & 
 \text{ \{Vector bundles over }\overline{_{K^p}\CS(G,X)}^{\rm min} \}  \\
\CE & \mapsto & [\CE]_p. \end{array}}
Similarly, the existence of the $K_h$ torsor $\mathcal{M}_{dR,\Sigma,K}$ over ${}_{K}\sS(G,X)_{\Sigma}$ gives rise to a functor
\ga{fdr}{   
\begin{array}{rcl}
{\rm Vect}^{\rm ss}_G(\hat{X}) [  ~~ \overset{r}\simeq  {\rm Rep}_{K_h}]  & \to & \text{ \{Vector bundles over }{}_{K}\sS(G,X)_{\Sigma} \}  \\
\CE & \mapsto & [\mathcal{E}^{\rm can}]. \end{array}}

With the notation of Theorem \ref{compar3}, we define 
\ga{}{  [\CE]_{p,\Ss,K_s} = Q^*  [\CE]_p, \ [\CE]_{dR,\Ss ,K_s}=\phi^*[\mathcal{E}^{\rm can}].}
 Let $Z^o_s = {}_{K^p}\CS(G,X)\times_{{}_{K^p}\CS(G,X)^{\rm min}} Z_s$, $\mathfrak{j}_s:  Z^o_s \hookrightarrow Z_s$ the canonical map.  
The isomorphism of torsors in Proposition \ref{compar} gives rise to a canonical isomorphism of tensor functors over $Z^o_s$
\begin{equation}\label{open}
( \CE \mapsto \mathfrak{j}_s^*([\CE]_{p,\Ss,K_s})) \isoarrow ( \CE \mapsto \mathfrak{j}_s^*( [\CE]_{dR,\Ss ,K_s})).
\end{equation}

  From Theorem~\ref{compar3} one immediately  obtains the following corollary. 

\begin{cor}\label{pullback} There is a canonical  isomorphism of tensor functors  
$$( \CE \mapsto [\CE]_{p,\Ss,K_s}) \isoarrow ( \CE \mapsto  [\CE]_{dR,\Ss ,K_s})$$
over $Z_s$, lifting the isomorphism \eqref{open}  over  $Z_s^o$.
\end{cor}
 
 We return to this theme in the next paragraph.  Recall briefly how the Hecke algebra acts on automorphic bundles on $ {} _KS(G,X)$. See \cite[Section 1.2]{EH17}.
  Fix $K \subset G(\af)$.  Let $h \in G(\af)$
 and consider $K_h = K \cap hKh^{-1} \subset K$.  
   Right multiplication by $h$ defines an isomorphism 
 $r_h:  {}_{hKh^{-1}}S(G,X)  \isoarrow {}_KS(G,X).$
 One defines
\ga{bdl}{T(h) [\mathcal{E}]_K=  r_{h*}\circ \pi_{ hKh^{-1}, K_h*} \circ \pi_{K, K_h}^* [\mathcal{E}]_K.}
The projection formula implies 
\ga{}{ T(h) [\mathcal{E}]_K= [\mathcal{E}]_K \otimes \pi_{K, K_h*} \mathcal{O}_{ {}_{K_h}S(G,X).    } \notag}
As one has pullbacks and push-downs for $K_{0,\mathbb{Q}}$ and for $\ell$-adic cohomology, one has on them an action of the Hecke algebra  $\mathcal{H}_K$, spanned by the $T(h)$  for $h  \in G(\af)$, with $h$ trivial at $p$ and at ramified places for $K$.  In particular,  formula \eqref{bdl} implies the formula
 \ga{vol}{ {\rm class \ of} \  T(h) [\mathcal{E}]_K = {\rm class \ of } \  [\mathcal{E}]_K^{[K:K_h]} \ {\rm in} \  K_0( {}_KS(G,X))_{\mathbb{Q}.}}
 Recall that the volume character of the Hecke algebra spanned by the $T(h)$ with values in a ring $R$, is the character sending $T(h)$ to $[K:K_h]$ viewed as an element in $R$ \cite[Lemma~1.15]{EH17}.  Thus 
 \eqref{vol} is saying that  classes of automorphic bundles in $K_0( {}_KS(G,X))_{\mathbb{Q}}$ are eigenvectors  for the volume character under the action of $\mathcal{H}_K$.

 \medskip

 Let $\ell \neq p$ be a prime number. 
 
 \begin{lem}
 The isomorphism $r_h$ induces an isomorphism  
 \ga{}{ r_h^{\rm min}:  {}_KS(G,X)^{\rm min}   \isoarrow {}_{hKh^{-1}}S(G,X)^{\rm min}.\notag}
 The finite morphism $\pi_{K, K_h}:  {}_{K_h}S(G,X) \to {}_KS(G,X)$ extends to a finite morphism
 \ga{}{\pi^{\rm min}_{K, K_h}:  {}_{K_h}S(G,X)^{\rm min} \to {}_KS(G,X)^{\rm min} \notag}
and the action of $\mathcal{H}_K$ on $H^{2n}({}_KS(G,X),\mathbb{Z}_\ell(n)) $   extends to an action on 
$H^{2n}({}_KS(G,X)^{\rm min},\mathbb{Z}_\ell(n))$.

 \end{lem}
 \begin{proof}  When $G = \GL(2)_\QQ$ the Shimura varieties are finite unions of modular curves, and the result in that case is standard.  We may thus assume $G^{\rm ad}$ contains no factor isomorphic to $\PGL(2)_\QQ$.   Then we know by Theorem~10.14 of \cite{BB66} that $H^0({}_KS(G,X), \omega^n)$ is finite dimensional for any $n \geq 0$, where 
$\omega$ is the dualizing sheaf.   The isomorphism  $r_h$ induces an algebra isomorphism  
\ga{}{ \oplus_{n\in \mathbb{N}} H^0({}_KS(G,X), \omega^n)  \isoarrow \oplus_{n\in \mathbb{N}}H^0(_{hKh^{-1}}S(G,X), \omega^n) \notag.}     Similarly the finite cover $\pi_{K, K_h}$ induces an injective algebra morphism
\ga{}{ \oplus_{n\in \mathbb{N}} H^0({}_KS(G,X), \omega^n)  \hookrightarrow \oplus_{n\in \mathbb{N}}H^0({}_{K_h}S(G,X), \omega^n) \notag.} 
Since ${}_KS(G,X)^{\rm min}$ (resp. ${}_{K_h}S(G,X)^{\rm min}$) is $\Proj$ of the left-hand side (resp. right-hand side) of the last diagram \cite[Theorem~10.11]{BB66}, the induced map on the $\Proj$ defines the extensions $r^{\rm min}_h$ and $\pi_{K,K_h}^{\rm min}$.
On the other hand, pullback on cohomology is defined, while the trace map ${\rm Tr}:  \pi_{K, K_h *} \mathbb{Z}_\ell\to \mathbb{Z}_\ell$ extends to 
$j_{K*}   \pi_{K, K_h*} \mathbb{Z}_\ell=  \pi^{\rm min}_{K, K_h *} j_{K_h*} \mathbb{Z}_\ell = \pi^{\rm min}_{K, K_h *} \mathbb{Z}_\ell \to j_{K*}\mathbb{Z}_\ell=\mathbb{Z}_\ell.$ This proves the second part.  
\end{proof}

 Let $E(G,X)=:E$ be the reflex field. Fix an embedding $E\hookrightarrow \bar{\QQ}$. We choose a place $w$ of $E$ dividing the prime $p$ which we assume to be different from  $\ell$. We extend $E\hookrightarrow \bar{\QQ}$ to the completion $E_w \hookrightarrow \bar{\QQ}_p$.  We denote by $G_{E_w}, G_E$ the Galois groups of $E_w, E$. We add a subscript on the lower right to indicate the field over which we regard the varieties considered. 
 We denote by   $H^i({}_KS(G,X)^{\rm min}_{ \bar{\QQ}}, \mathbb{Q}_\ell)_v =H^i({}_KS(G,X)^{\rm min}_{ \bar{\QQ}_p}, \mathbb{Q}_\ell)_v
 $ the eigenspace under the volume character of the action of $\mathcal{H}_K$.

Through the end of this section, we  set $K_0=K$  for the neat subgroup $K\subset  G(\af)$  and use
 the notation $H^i(\Lambda(j))$ for \'etale cohomology with coefficients $\Lambda$,  which are either $\mathbb{Z}_\ell$ or $\mathbb{Q}_\ell$. 
\begin{thm}\label{chernp}   Let $(G,X)$ be a Shimura datum of Hodge type. For any prime $\ell \neq p$, any $\CE \in {\rm Vect}^{\rm ss}_G(\hat{X}_E)$, we define  $\ell$-adic Chern classes
\ga{}{
c_{n, \ell}([\CE]_p){}_{K_r} \in H^{2n}\left({}_{K_r}S(G,X)^{\rm min}_{\bar{\QQ}_p}, \Lambda(n)\right)_v^{G_{E_w}},  \   \ n  \  \in \ \mathbb{N}\notag}
where $\Lambda=\mathbb{Q}_\ell$ for $r=0$ and $\mathbb{Z}_\ell$ for $r\ge 1$, with
the following properties:

\begin{itemize}
\item[\rm (1)]  $ j_{K_r}^*(c_{n,\ell}([\CE]_p){}_{K_r}) = c_{n, \ell}([\CE]_{K_r}) \in H^{2n}(_{K_r}S(G,X)_{\bar{\QQ}},\Lambda(n)),$
where the right hand side is the $\ell$-adic Chern class   of $[\CE]$ on $_{K_r}S(G,X)$.
\item[\rm (2)]  If $K' \subset K$ then $\pi^{\rm min *}_{K_r,K'_r}(c_{n, \ell}([\CE]_p){}_{K_r})=c_{n, \ell}([\CE]_p){}_{K_r'}$.
\item[\rm (3)]   $r_h^{\rm min *}(c_{n,\ell}([\CE]_p){}_{K_r}) = c_{n,\ell}([\CE]_p)_{hK_rh^{-1}}$ for $h\in G(\af)$ with trivial component at $p$ and at places ramified for $K$.
\item[\rm (4)]   Whitney product formula: If $0\to  \mathcal{E}_1\to  \mathcal{E} \to \mathcal{E}_2\to 0$ is an exact sequence in ${\rm Vect}^{\rm ss}_G(\hat{X}_E)$, then 
$$c_{n,\ell}([\CE]_p){}_{K_r})=\oplus_{a+b=n} c_{a,\ell}([\CE _1]_p){}_{K_r}) \cdot  c_{b,\ell}([\CE _2]_p){}_{K_r}).$$
\item[\rm (5)] If $(G',X')\to (G, X)$ is a morphism of Shimura data of Hodge type,  $K'\to K$ is a compatible level, with $K_r$ and $K_r'$ compatibly chosen, and
$$f:  {}_{K_r'}S(G',X')^{\rm min} \ra {}_{K_r}S(G,X)^{\rm min}$$
the corresponding morphism of Shimura varieties, 
then  $$f^* (c_{n,\ell}([\CE]_p){}_{K_r}) =(c_{n,\ell}([f^*\CE]_p){}_{K_r'})  \in
 H^{2n}({}_{K_r'}S(G',X')^{\rm min}_{\bar{\QQ}_p},\Lambda(n))_v^{G_{E_w}}.$$
\end{itemize}
\end{thm}

\begin{proof} 
We first construct $c_{n, \ell}([\CE]_p){}_{K_r}$ without the $v$ invariance. The vector bundle $[\CE]_p$ from \eqref{fp} has Chern classes 
$$\bar{c}_{n, \ell}( [\CE]_p) \in \varprojlim_m H^{2n}(\overline{{}_{K^p}\mathcal{S}(G,X)}^{\rm min},\mathbb{Z}/\ell^m (n)), \ n\in \mathbb{N}$$ by Proposition~\ref{p1bundle}. On the other hand, $[\CE]_p$ is defined as $\pi_{HT}^*\mathcal{E}$,  where $\mathcal{E}$ is viewed as a
vector bundle on the $C$-adic space $\hat {\mathcal{X}}$. But $\mathcal{E}$  is already defined on $\hat{X}_E$. In particular, with $K_{p,r}$ defined as in 
 Section~\ref{s:HT}, the classes 
$\bar{c}_{n, \ell}( [\CE]_p)$ are invariant under the action of $K_{p,r}\subset G(\mathbb{Q}_p)$  
for all $r$ and by $G_{E_w}$. By Proposition~\ref{etaleinvariant} it follows that the classes  
$\bar{c}_{n, \ell}( [\CE]_p) $  uniquely define classes
$$ c_{n, \ell}( [\CE]_p) \in H^{2n}(  \overline{{}_{K_r} S(G,X)}^{\rm min}_{\bar{\mathbb{Q}}_p},\mathbb{Z}_\ell(n))^{G_{E_w}}$$ 
for $r\ge 1$.  Since $K/K_1$ is finite, we can descend to $\mathbb{Q}_\ell$-cohomology at level $K$. 
We thus have descended classes
$$\bar{c}_{n, \ell}([\CE]_p){}_{K_r} \in H^{2n}(  \overline{{}_K S(G,X)}^{\rm min}_{\bar{\mathbb{Q}}_p},\mathbb{Q}_\ell(n))^{G_{E_w}}.$$
Letting $u_r:  {}_{K_r} S(G,X)^{\rm min} \ra \overline{{}_{K_r} S(G,X)}^{\rm min}$ denote the natural map, as in \eqref{composite}, the classes
$$c_{n, \ell}([\CE]_p){}_{K_r} := u_r^*(\bar{c}_{n, \ell}([\CE]_p){}_{K_r})$$
are the desired ones.

Property (1) is then a direct consequence of Proposition~\ref{compar}. Property (2)  follows directly from the construction using descent via Proposition~\ref{etaleinvariant}.  As for Property (3), using descent again, it is enough to see that $r_h^{\rm min}$ extends to an isomorphism  $$ r^{\rm perf, min}_h:
{}_{K_r}\mathcal{S}(G,X)^{\rm min} \to {}_{hK_rh^{-1}}\mathcal{S}(G,X)^{\rm min}$$ of perfectoid spaces, which respects $[\CE]_p$. (See \cite[Theorem 4.1.1\,(iv)]{Sch15} for the Siegel modular case; the general case is identical.)  Our notation $[\CE]_p$ does not refer to $K_p$. One should replace  the notation
$[\CE]_p$ with
 $[\CE]_p^{K_p}$. Then the compatibility means $r_h^{\rm perf, min*} [\CE]_p^{hK_ph^{-1}}=[\CE]^{K_p}_p$, which follows from Theorem~\ref{perfec2} together with the addendum Theorem~\ref{thmCS}.
 Property (4) follows directly from the Whitney product formula for the Chern classes of $[\mathcal{E}]_p$ and Property (5) of the functoriality of the construction of ${}_{K^p}\mathcal{S}(G,X)^{\rm min}$ (\cite{Sch15}, {\it loc.cit.}, comes from the construction, even if this is not explicitly mentioned). 
 It remains to see  that $c_{n, \ell}([\CE]_p){}_{K_r} \in H^{2n}({}_{K_r}\mathcal{S}(G,X)^{\rm min}_{\bar{\mathbb{Q}}_p},\Lambda(n))_v$.  
By construction, this follows from  $$\bar{c}_{n, \ell}( [\CE]_p) \in \varprojlim_m H^{2n}\left(\overline{{}_{K^p}\mathcal{S}(G,X})^{\rm min},\mathbb{Z}/\ell^m (n)\right)_v, \ n\in \mathbb{N}, $$
which again follows by the projection formula from the trivial relation
\ga{}{ T(g) [\mathcal{E}]_p= [\mathcal{E}]_p \otimes_{ \mathcal{O}_{ {}_{K^p}\CS(G,X)^{\rm min}} }    \mathcal{O}_{ {}_{K_h^p}\CS(G,X)^{\rm min}.    } \notag}
\end{proof}

\begin{nota} \label{nota:chl}
{\rm The functor 
\ga{}{
\left\{\begin{array}{rcl}
{\rm Rep}_{\bar{ \mathbb{Q}}}(K_h)   & \to & 
\{{\rm Vector \ Bundles \ on} \ {}_{K^p} \mathcal{S}(G, X)^{\rm min}\} \\
 \mathcal{E} & \mapsto &  [\mathcal{E}]_p 
\end{array}\right. ,\notag}
is a tensor functor,  which sends 
 ${\rm Rep}_{\bar{ \mathbb{Q}}}(G)$ to trivial bundles. 
It
induces the Chern character 
\ml{}{ K_0({\rm Rep}_{\bar{ \mathbb{Q}}}(K_h)  \otimes_{ {\rm Rep}_{\bar{ \mathbb{Q}}}(G)}\mathbb{Q}) \xrightarrow{ch_\ell}
H^{2*}\left({}_KS(G,X)^{\rm min}_{\bar{\mathbb{Q}}_p}, \mathbb{Q}_\ell(*)\right)^{G(E_w)}_v  \subset 
\left(\varprojlim_m H^{2*}\left({}_{K^p}\mathcal{S}(G,X)^{\rm min}, \mathbb{Z}/\ell^m(*)\right)\right) \otimes \QQ. \notag  }
}
\end{nota}

\subsection{The image of $H^*({}_KS(G,X)^{\rm min})$ in $H^*({}_KS(G,X)_\Sigma)$}

We now return to the topological setting.  In this section, ${}_KS(G,X)^{\rm min}$ and ${}_KS(G,X)_\Sigma$ 
are identified with the analytic spaces  underlying their $\CC$-valued points.  We fix $\Sigma$,
 use the notation~\eqref{2.1}
 and let $\varphi := \varphi_\Sigma$ denote the  desingularization map $\varphi_\Sigma:  {}_KS(G,X)_\Sigma \ra {}_KS(G,X)^{\rm min}$.  
 It is a projective morphism.  Recall from the general theory \cite[Theorem~2.8.1]{dCM05} that the choice of a polarization  $\mathcal{L}$ for $\varphi$ induces a factorization 
  \ga{fact}{ 
\vcenter{\vbox{\xymatrix{ \ar[drr]_{\varphi^*} H^i\left({}_KS(G,X)^{\rm min},\QQ\right) \ar[rr]^{\rm natural \ map} &  & IH^i\left({}_KS(G,X)^{\rm min},\QQ\right)    \ar@{^{(}->}[d]^{\varphi^*_{\mathcal L} }\\  
&  & H^i\left({}_KS(G,X)_\Sigma,\QQ\right) } 
}
}}
in which all morphisms are compatible with the polarized Hodge structure, and $\varphi_{\mathcal{L}}^*$ is injective  (see e.g. \cite[Corollary~2.8.2]{dCM05}).
For a prime $\ell \neq p$ and any $s\ge 0$, the diagram \eqref{fact} has an $\ell$-adic version
   \ga{factl}{ 
\vcenter{\vbox{
\xymatrix{ \ar[drr]_{\varphi^*} H^i\left({}_{K_s}S(G,X)^{\rm min},\mathbb{Q}_\ell (j)\right) \ar[rr]^{\rm natural \ map} &  & IH^i\left({}_{K_s}S(G,X)^{\rm min},\mathbb{Q}_\ell (j)\right)    \ar@{^{(}->}[d]^{\varphi^*_{\mathcal L} }\\  
&  & H^i\left({}_{K_s}S(G,X)_\Sigma,\mathbb{Q}_\ell (j)\right)  } 
}
}}
Thus by the comparison isomorphism   \cite[Section~6]{BBD82}, we conclude that $\varphi_{\mathcal{L}}^*$ in the diagram~\eqref{factl} is injective as well.
The image of $\varphi_{\mathcal{L}}^*$ might depend on the choice of $\mathcal{L}$, see  \cite[Example 2.9]{dCM05}.  Furthermore,  $IH^i({}_KS(G,X)^{\rm min},\QQ), $ thus 
$$\varphi_{\mathcal{L}}^*\left(IH^i({}_KS(G,X)^{\rm min},\QQ)\right)\subset  H^i({}_KS(G,X)_\Sigma,\QQ),  $$ 
is pure of weight $i$,  and 
$\varphi^*( H^i({}_KS(G,X)^{\rm min},\QQ)) \subset 
H^i({}_KS(G,X)_\Sigma,\QQ) $ is the maximal pure weight $i$ quotient  of $H^i({}_KS(G,X)^{\rm min},\QQ)$ by \cite[Proposition~8.2.5]{Del74}.

\begin{thm}\label{pullback_Chern} For any $\CE \in {\rm Vect}^{\rm ss}_G(\hat{X})$, any neat level subgroup $K \subset G(\af)$ and all $s \ge 0$, we have 
\ga{}{\varphi^*((c_{n,\ell}([\CE]_p){}_{K_s}) = c_{n,\ell}(( [\CE]^{\rm can}){}_{K_s}) \in H^{2n}\left(
{}_{K_s}S(G,X)_{\Sigma, \bar{\mathbb{Q}}}, \Lambda(n)\right)^{G_E}. \notag}
Here $\Lambda=\mathbb{Q}_\ell$ for $s =0$ and $\mathbb{Z}_\ell$ for $s \ge 1$, as in Theorem \ref{chernp}.  
\end{thm}
\begin{proof}   We place ourselves in the situation of Theorem~\ref{compar3} which we use; in particular we fix the natural number $s$.  The proof assuming Hypothesis \ref{projtor} is  the same without fixing $s$.  
It suffices to check the claims with $c_{n,\ell}([\CE]_p){}_{K_s}$ replaced by $\bar{c}_{n,\ell}([\CE]_p){}_{K_s}$.

Notation is as in Section \ref{sec:covers}.    For $s \geq 0$ we have the  diagram with commutative squares

\ga{W2}{ 
\vcenter{\vbox{
\xymatrix{ 
& \ar@/^2pc/[rr]^Q  \ar[d]_{\phi} Z_s \ar[r]^{\psi} &  \overline{\ar[d]_{ \pi_{K_s, K^p}}    {}_{K^p}\sS(G,X)_{\Sigma}} \ar[r]^{\bar q} & \ar[d]^{\pi_{K_s, K^p}^{\rm min}} 
\overline{{}_{K^p}\sS(G,X)}^{\rm min}   \\
&_{K_s} \sS(G,X)_\Sigma \ar[r]_\g   & \overline{ _{K_s}\sS(G,X)_\Sigma  } \ar[r]_\delta & \overline{_{K_s} \sS(G,X)}^{\rm min}} }}  }
Here $\delta \circ \gamma = \varphi$ and $\delta, \bar q$ come from the definitions~\eqref{bar} and \eqref{bar2}.

 By Theorem~\ref{compar3} one has  
\begin{align}
 (\pi_{K_s, K^p}^{\rm min}\circ Q)^* \left(\bar{c}_{n,\ell} ([\mathcal{E}]_p))_{K_s}\right) & =
 \phi^* (\delta \circ \gamma)^* (\bar{c}_{n,\ell} ([\mathcal{E}]_p))_{K_s}) 
\notag \\
& =   Q^* (\bar{c}_{n,\ell} ([\mathcal{E}]_p)) =   \phi^*(\bar{c}_{n,\ell}( [\CE]^{\rm can})) 
\notag \\   
& \in Q^* H^{2n}\left({}_{K^p}\CS(G,X)^{\rm min}, \Lambda(n)\right)   \subset H^{2n}( Z_s,\Lambda(n)).
\notag
\end{align}
Since \eqref{W2} commutes, and since $\phi^*$ is injective by Proposition \ref{CY}, we thus have
\begin{equation}\label{confirm} c_{n,\ell}( [\CE]^{\rm can}) = ( \delta \circ \gamma)^*(\bar{c}_{n,\ell} ([\mathcal{E}]_p))_{K_s}) = \varphi^*((\bar{c}_{n,\ell}([\CE]_p){}_{K_s}),
\end{equation}
as classes in  
$H^{2n}({}_{K_s}\CS(G,X)_{\Sigma, \bar{\mathbb{Q}}}, \Lambda(n)).$
On the other hand, 
$$H^{2n}\left({}_{K_s}\CS(G,X)_{\Sigma, \bar{\mathbb{Q}}}, \Lambda(n)\right)= H^{2n}\left({}_{K_s} S(G,X)_{\Sigma, \bar{\mathbb{Q}}}, \Lambda(n)\right)$$
by the comparison isomomorphism \cite[Theorem~4.2]{Hub98}.
In addition, the automorphic bundles $[\CE]^{\rm can}$ are defined over $E$. Thus 
 the classes in \eqref{confirm}   lie in
$H^{2n}({}_{K_s} S(G,X)_{\Sigma, \bar{\mathbb{Q}}}, \Lambda(n))^{G_E}$. 
 This finishes the proof. 
 \end{proof}

\medskip

Let 
\ga{}{H^{2*}_{\rm Chern}\left({}_{K}S(G,X)^{\rm min}_{\bar{\mathbb{Q}}_p},\mathbb{Q}_\ell(*)\right) \subset H^{2*}\left({}_{K}S(G,X)^{\rm min}_{\bar{\mathbb{Q}}_p},\mathbb{Q}_\ell(*)\right) \notag}  denote the $\mathbb{Q}_\ell$-subalgebra generated by the images of $c_{*,\ell}([\CE]_p){}_K$, and 
\ga{}{H^{2*}_{\rm Chern}({}_KS(G,X)_{\Sigma,\bar{\mathbb{Q}}_p},\mathbb{Q}_\ell(*)) \subset H^{2*}(_{K}S(G,X)_{\Sigma, \bar{\mathbb{Q}}_p}\mathbb{Q}_\ell(*)) \notag} 
denote the $\mathbb{Q}_\ell$-subalgebra generated by the images of $c_{*,\ell}([\CE]^{\rm can}_K){}_K$.  Here, when $G$ satisfies (i) of Definition \ref{plusdef}, we take $\CE$ to belong to ${\rm Vect}_G(\hat{X})^+$.  

\begin{prop} \label{Chern} The commutative diagram \eqref{factl} restricts to the commutative diagram of $\mathbb{Q}_\ell$-algebras
  \ga{eq:chern}{ 
\vcenter{\vbox{
\xymatrix{
    & &\ar[ddll]_{{\rm ch}_\ell \ \sim} \left[K_0(\Rep_\Qbar(K_h))\otimes_{K_0(\Rep_\Qbar(G))} \QQ_{\ell}\right]^+ \ar[d]^{ \sim (\rm  ch~\circ\eqref{01} )} \\
    & & H^{2*}\left(\hat X_{\bar{\mathbb{Q}}_p}, \mathbb{Q}_\ell(*)\right)^+ \ar[d]^{\sim ({\rm Prop}.~\ref{isomQ})}\\
   \ar[drr]_{\varphi^* \ \sim } H^{2*}_{\rm Chern}\left({}_KS(G,X)^{\rm min}_{\bar{\mathbb{Q}}_p},\QQ_\ell (*)\right) \ar[rr]^{\rm natural \ map  \  \sim}  &  & IH^{2*}({}_KS(G,X)^{\rm min}_{\bar{\mathbb{Q}}_p},\QQ_\ell(*))_v^+   \ar[d]^{\varphi^*_{\mathcal L} \ \sim}\\  
 &  & H^{2*}_{\rm Chern}\left({}_KS(G,X)_{\Sigma, \bar{\mathbb{Q}}_p},\QQ_\ell(*)\right)} 
}}
}
 in which $\varphi_{\mathcal L}^*$, ${\rm ch}_\ell$,   the natural map  and $\varphi^*$  are isomorphisms of $\mathbb{Q}_\ell$-algebras.
In particular 
\ga{}{\varphi_{\mathcal{L}}^*\left( IH^{2*}({}_KS(G,X)^{\rm min}_{\bar{\mathbb{Q}}_p},\QQ_\ell(*))^+_v\right)\subset  H^{2*}\left({}_KS(G,X)_{\Sigma,\bar{\mathbb{Q}}_p},\QQ_\ell(*)\right) \notag}  is a $\mathbb{Q}_\ell$-sub-vectorspace which does not depend on the choice of the relative polarization $\mathcal{L}$.
\end{prop}

\begin{proof} 
By Theorem \ref{gorpar2}, the image of the Goresky-Pardon Chern classes $c_n([\CE]_K)^{GP}$ in the algebra
$IH^{2*}({}_KS(G,X)^{\rm min},$ $\mathbb{Q}(*))$  lies in $IH^{2*}({}_KS(G,X)^{\rm min},  \mathbb{Q}(*))^+_v$ and spans this  subalgebra over $\mathbb{Q}$. In particular, the image of
$$H^{2*}\left({}_KS(G,X)^{\rm min},  \mathbb{Q}(*)\right)_v \xrightarrow{\rm natural \ map} IH^{2*}\left({}_KS(G,X)^{\rm min},  \mathbb{Q}(*)\right)_v$$ contains $IH^{2*}({}_KS(G,X),  \mathbb{Q}(*))^+_v$, and the map
$$ \varphi_{\mathcal{L}}^*:  IH^{2*}\left({}_KS(G,X),  \mathbb{Q}(*)\right)^+_v \to  H^{2*}_{\rm Chern}\left({}_KS(G,X)_{\Sigma},\QQ(*)\right)^+$$
is surjective.
By comparison with $\ell$-adic intersection cohomology, the image of
$$H^{2*}\left({}_KS(G,X)_{\bar{\mathbb{Q}}_p},  \mathbb{Q}_\ell(*)\right)_v \xrightarrow{\rm natural \ map} IH^{2*}\left({}_KS(G,X)_{\bar{\mathbb{Q}}_p},  \mathbb{Q}_\ell(*)\right)_v$$
contains $IH^{2*}({}_KS(G,X)_{\bar{\mathbb{Q}}_p},  \mathbb{Q}_\ell(*))_v^+$, and 
$$ \varphi_{\mathcal{L}}^*:  IH^{2*}\left({}_KS(G,X)_{\bar{\mathbb{Q}}_p},  \mathbb{Q}(*)\right)^+_v \to  H^{2*}_{\rm Chern}\left({}_KS(G,X)_{\Sigma, {\bar{\mathbb{Q}}_p}},\QQ(*)\right)$$
is surjective as well.

We have no direct way to compare the Goresky-Pardon classes to our classes from Theorem~\ref{chernp} on ${}_KS(G,X)_{\bar{\mathbb{Q}}_p}$. However,  applying
Proposition~\ref{pullback_Chern} for $s=0$, i.e. for $K_0=K$,  one has 
\ga{}{\varphi_{\mathcal{L}}^* \left(H^{2*}_{\rm Chern}\left({}_KS(G,X)^{\rm min}_{\bar{\mathbb{Q}}_p},\QQ_\ell (*)\right)\right)= 
H^{2*}_{\rm Chern}\left({}_KS(G,X)_{\Sigma, \bar{\mathbb{Q}}_p},\QQ_\ell(*)\right).\notag} 
As $\varphi_{\mathcal{L}}^*$ is injective, we  deduce that
the natural map sends 
 the $\mathbb{Q}_\ell$-subalgebra $H^{2*}_{\rm Chern}({}_KS(G,X)^{\rm min}_{\bar{\mathbb{Q}}_p},\QQ_\ell (*)) $  onto 
$IH^{2*}({}_KS(G,X)^{\rm min}_{\bar{\mathbb{Q}}_p},\QQ_\ell(*))^+_v $.  This proves that the natural map in the proposition and $\varphi^*$ are surjective. 
It remains to prove that they are injective as well.  The Chern character has by definition 
values in $H^{2*}_{\rm Chern}({}_KS(G,X)^{\rm min}_{\bar{\mathbb{Q}}_p},\QQ_\ell (*))$ and the whole diagram commutes. This proves that $\varphi^*$ is injective. It follows that  the  homomorphisms ${\rm ch}_\ell, \ \varphi^*$ and the natural map are isomorphisms. 
Finally, the image of $\varphi^*$ being equal to the image of $\varphi^*_{\mathcal{L}}$, the latter does not depend on $\mathcal{L}$. This finishes the proof. 
\end{proof}

\begin{lem} \label{lem:flat}
The classes
\ga{}{
c_{n, \ell}([\CE]_p){}_{K_r} \in H^{2n}\left({}_{K_r}S(G,X)^{\rm min}_{\bar{\QQ}_p}, \QQ_\ell(n)\right)_v^{G_{E_w}} \notag}
 constructed in Theorem~\ref{chernp} vanish for all $n>0$ when $\mathcal{E}$ is flat, that is when $\mathcal{E}$ comes from a $G$-representation.
\end{lem}

\begin{proof}
Under the assumption of the lemma,  $[\mathcal{E}]_p$ is trivial (see Notation~\ref{nota:chl}). 
The construction of Theorem~\ref{chernp} thus implies the lemma. 
\end{proof}

\begin{rmks}\label{abeliantype} {\rm (i) If we ignore the action of the Galois group, the Chern classes constructed above depend only on the connected components of ${}_KS(G,X)^{\rm min}$.  Thus the results above extend to general Shimura varieties of abelian type.  We leave the precise formulation to the reader.

(ii)  We naturally expect that Proposition~\ref{Chern}  remains true without the superscripts $^+$, which in any case change nothing when $G$ satisfies condition (iv) of Theorem \ref{gorpar}.  In the excluded case, it is easy to see that the classes denoted $e_r$ in \S \ref{evenquad} have non-trivial images in $IH^{n}({}_KS(G,X)^{\rm min}_{\bar{\mathbb{Q}}_p},\QQ_\ell(*))$; indeed, this follows from the results of \cite{GP02} and the equality $e_r^2 = c_{1,r}^{n}$ for each $r$.  However, we also have
$$(c_{1,r}^{k-1})^2 = c_{1,r}^n,$$
so it is possible that $c_{1,r}$ and $\pm e_r$ have the same image in $IH^{n}$.  We have not attempted to prove that this is not the case.
}
\end{rmks}

\section{Comparison of torsors}\label{proofofcomparison}

Let $(G,X) = (\GSp(2g),X_{2g})$, write $\hX_{2g}$ for the compact dual of $X_{2g}$, fix a base point $h \in X_{2g} \subset \hX_{2g}$, and let $\Omega_h$ denote the fiber at $h$ of the cotangent bundle to $\hX_{2g}$.   
Let $V_g$ denote the vector group $\mathbb{G}_a^g$.    Let $K_{h,g} \subset G$ denote the stabilizer of $h$.
Then $K_{h,g}$ can be identified with $\GL(g) \times \mathbb{G}_m = {\rm Aut}(V_g)\times \mathbb{G}_m$ in such a way that the restriction to $\GL(g)$ of the isotropy action of $K_{h,g}$ on $\Omega_h$ is equivalent to ${\rm Sym}^2(V_g)$. 
 We take $\St$ to be the standard representation $K_{h,g} \mapsto {\rm Aut}(V_g)$, and let $\CE_{\St}$ be the corresponding equivariant vector bundle on $\hX_{2g}$.  
 Fix $K_g = K_{g,p}\cdot K^p_g$ a neat open compact subgroup of $\GSp(2g)(\af)$, and fix a toroidal datum $\Sigma$ for $K$.  As in \S 2, we write  $$\mathfrak{q}_g:  _{K^p_g}\mathcal{A}_g^{\rm tor} \ra _{K^p_g}\mathcal{A}_g^{\rm min}$$
 for the morphism denoted $q$ above.  
\begin{prop}[\cite{PS16}]  There is a canonical isomorphism
$$\theta_g: \mathfrak{q}_g^*[\CE_{\St}]_{p} \isoarrow [\CE_{\St}]_{dR,\Sigma}$$
of vector bundles over  $_{K^p_g}\mathcal{A}_g^{\rm tor}$.
\end{prop}
\begin{proof}  This is essentially equivalent to Corollaire~1.16  of \cite{PS16}.  More precisely, the generic fiber of the bundle denoted $\omega^{\rm mod}_A$ in \cite{PS16} is exactly $[\CE_{\St}]_{dR,\Sigma}$.  Indeed, with our notation, the restriction of  $[\CE_{\St}]^{\rm can}$ to the open Siegel modular variety in (neat) finite level $K$ is the sheaf of relative sections of  the relative $1$-forms of the universal polarized abelian scheme, and its canonical extension to the toroidal compactification is isomorphic to the sheaf of relative sections of the  relative invariant $1$-forms on the corresponding semi-abelian scheme; see \cite[Proposition 6.9]{Lan12} or \cite[(1.3.3.15)]{Lan17}; the comparison between the algebraic and analytic constructions is Theorem 5.2.12 of \cite{Lan12a}.  
\end{proof}

We  now restrict  the information  to Shimura varieties of Hodge type.
Let $(G,X)$ be a Shimura datum of Hodge type, with compact dual $\hX$.  
We fix a symplectic embedding 
\ga{}{\iota:  (G,X) \ra (\GSp(2g),X_{2g}).\notag}
Let $h \in X$ be a base point, let $K_h \subset G$ be its stabilizer,  let $K_{h,g} \subset \GSp(2g)$ be the stabilizer of $\iota(h)$, and denote by $\iota_h$ the inclusion of $K_h$ in 
$K_{h,g}$.    
Let $\St_h = \St\circ \iota_h:  K_h \ra {\rm Aut}(V_g)$, with $\St$ the standard faithful representation of $K_{h,g}$.
Finally, let 
$\CE_{\St_h}$ be the  equivariant vector bundle on $\hX$ with isotropy representation $\St_h$.   Let $K = K_p\cdot K^p$ be a neat open compact subgroup of $G(\af)$.

\begin{cor}\label{1forms}   We fix a natural number $s$.
The morphism $\theta_g$ induces  by pullback for any $s$ a canonical isomorphism  
$$\theta(s):   Q^*[\CE_{\St_h}]_{p} \isoarrow [\CE_{\St_h}]_{dR,\Sigma, K_s}$$
of vector bundles over $Z_s$.

\end{cor}
\begin{proof}[Proof of 
Theorem~\ref{compar3}]  We fix $s$ as in Section~\ref{sec:covers}.
As $\St_h$ is faithful,  any  irreducible representation $V$ of $K_h$ defined over $\bar{\mathbb{Q}}$ is a direct factor of the  representation $\St_h^{ \otimes m}\otimes \St_h^{ \vee \otimes n} $ for some pair of natural numbers $(m,n)$ \cite[I, Proposition~3.1\,(a), II, Proposition~2.23]{DMOS82}.  Write  
$$V \xrightarrow{\sigma} \St_h^{ \otimes m}\otimes \St_h^{ \vee \otimes n}   \xrightarrow{\tau} V$$ for the splitting. 
The isomorphism of vector bundles  $\theta(s)$  induces an isomorphism of vector bundles 
$$\theta( s)^{\otimes m}\otimes \theta(s)^{\vee \otimes n}:   \mathfrak{q}_g^*\left([\CE_{\St_h}]^{\otimes m}_{p} \otimes 
[\CE_{\St_h}]^{ \vee \otimes n}_{p}\right)
 \isoarrow  
[\CE_{\St_h}]_{dR,\Sigma,K_s}^{\otimes m} \otimes [\CE_{\St_h}]_{dR,\Sigma,K_s}^{\vee \otimes n}    
$$
 while $\sigma$ and $\tau$ induce the splittings of vector bundles over  ${}_{K^p}\sS(G,X)^{\rm min} $
   and over   $Z_s$ 
\ga{}{  [\mathcal{V}]_p \xrightarrow{\sigma_p} [\CE_{\St_h}]^{\otimes m}_{p} \otimes 
[\CE_{\St_h}]^{ \vee \otimes n}_{p} \xrightarrow{\tau_p} [\mathcal{V}]_p \notag\\
[\mathcal{V}]_{dR, \Sigma,K_s} \xrightarrow{\sigma_{dR, \Sigma,r}} [\CE_{\St_h}]^{ \otimes m}_{dR, \Sigma,K_s} \otimes 
[\CE_{\St_h}]^{  \vee \otimes n}_{p} \xrightarrow{\tau_{dR,\Sigma,K_s}} [\mathcal{V}]_{dR,\Sigma,K_s}. \notag }
This yields the following morphisms of coherent sheaves   over $Z_s$ 
\ga{}{\alpha:= \tau_{dR,\Sigma,K_s} \circ \theta(s)^{\otimes m}\otimes \theta(s)^{\vee \otimes n} \circ  Q^*\sigma_p: 
 Q^*[\mathcal{V}]_p\to [\mathcal{V}]_{dR,\Sigma,K_s} \notag\\
\beta:=   Q^*\tau_p\circ (\theta(s)^{\otimes m}\otimes \theta(s)^{\vee \otimes n})^{-1} \circ  \sigma_{dR, \Sigma,K_s}:
[\mathcal{V}]_{dR,\Sigma,s} \to   Q^*[\mathcal{V}]_p \notag
}
which pull back to inverse isomorphisms  on $Z^o_s$ under $\mathfrak{j}_s$ (see \eqref{open}).
Thus 
$$
\beta\circ \alpha -{\rm Id}_{ Q^*[\mathcal{V}]_p }:   Q^*[\mathcal{V}]_p \to  Q^*[\mathcal{V}]_p, \quad \alpha \circ \beta - {\rm Id}_{  [\mathcal{V}]_{dR,\Sigma, K_s} }:  [\mathcal{V}]_{dR,\Sigma,K_s}
 \to  [\mathcal{V}]_{dR,\Sigma,K_s}
$$
are homomorphisms of coherent sheaves  which pull back to zero  on $Z^o_s$. As $ Q^*[\mathcal{V}]_p$ and  $[\mathcal{V}]_{dR,\Sigma,K_s}$ are vector bundles, thus are both torsion free, we conclude that both maps are $0$.  

This finishes the  proof of Theorem \ref{compar3}.  
As for Theorem~\ref{compar2}, the proof is identical, except that we do not have to fix $s$ in the beginning and we replace $Q$ by $q$ everywhere.
\end{proof}

\appendix{A}{Geometry} \label{sec:app}
\numberwithininappendix

We collect here the statements of some standard results in the theory of schemes, in versions adapted to perfectoid spaces.   The proofs are due to Peter Scholze.

\begin{prop}[Scholze] \label{p1bundle}  Let $\CW$ be a vector bundle of rank $r$ over a perfectoid space $\CX$over $C$.   Let $\PP(\CW)$ denote the corresponding projective bundle over $\CX$, viewed as an adic space over $C$ \emph{(}see Remark \ref{sheafy} below\emph{)}. 
Then
$\mathcal{O}_{\PP(\CW)}(1)$  has a first  Chern class $$(z_m) \in \varprojlim_m H^2\left(\PP(\CW),\mathbb{Z}/\ell^m (1)\right)$$ such that the homomorphism 
\ga{}{\varprojlim_m H^b(\PP(\CW),\mathbb{Z}/\ell^m(a)) \leftarrow \oplus_{i=0}^{r-1} \varprojlim_m H^{b-2i}(\CX, \mathbb{Z}/\ell^m (a-i))\cdot (z_n)^i \notag} 
is an isomorphism. 
\end{prop}

\begin{rmk}\label{sheafy}  {\rm One of the referees asked how $\PP(\CW)$ can be defined as an adic space.   We forwarded the question to Laurent Fargues and David Hansen, who sent essentially identical replies.  Here is an (almost literal) paraphrase of Hansen's argument.  First, without loss of generality, we may localize on $\CX$ and assume that $\CX$ is affinoid perfectoid, say $\mathrm{Spa}(R,R^+)$, and that $\CW$ is the trivial bundle.   Thus we need to define $\mathrm{Spa}(R,R^+)\times \PP^{r-1}$ as an adic space.  Next, by the usual covering of $\PP^{r-1}$ by $r$ affinoid balls, we are reduced to showing that the ring $R\langle T_1, \dots, T_{r-1} \rangle$ (power series with coefficients in $R$ that tend to $0$) is {\it sheafy} \cite[Definition 3.1.11]{SW17}.
 In fact, $R\langle T_1, \dots, T_{r-1} \rangle$ is {\it sousperfectoid}, in the terminology of Hansen and Kedlaya, and in particular is sheafy; see \S 6.3 of \cite{SW17} for the proofs.  Moreover, $R\langle T_1, \dots, T_{r-1} \rangle$ is {\it analytic} \cite[Definition 4.4.2]{SW17}.

We are warned, however, that there is as yet no reference for the \'etale cohomology in the generality of sousperfectoid spaces.  We have seen that $\PP(\CW)$ is an analytic adic space; thus, in order to define the \'etale cohomology spaces used in Proposition \ref{p1bundle}, one can pass to the associate locally spatial diamond (see \cite[\S 15]{Sch17}), and  apply the cohomological formalism developed in \cite{Sch17}.}
\end{rmk}

We define the Chern class $c_i(\CW)$ by Grothendieck's  standard equation
$$
\varprojlim_m H^{2r}( \PP(\CW),\mathbb{Z}/\ell^m(r))  \ni \sum_{i=0}^r (-1)^i c_i(\CW) \cdot (z_m)^{r-i}=0, \quad (c_i(\CW)_m) \in \varprojlim_m H^{2i}(\PP(\CW),\mathbb{Z}/\ell^m(i))  .
$$

\begin{proof}
One defines $z$ as usual, as the projective limit over $m$ of the images $z_m\in H^2(\PP(\CW),\mathbb{Z}/\ell^m (1))$   of  the class of $ \mathcal{O}_{\PP(\CW)}(1)$ in $H^1( \PP(\CW), \mathbb{G}_m)$  via the connecting homomorphism of the \'etale Kummer exact sequence  $1\to \mu_{\ell^m}\to 
\mathbb{G}_m \xrightarrow{\ell^m} \mathbb{G}_m\to 1$. To prove the statement, it is enough to prove that  the map
\ga{}{H^b(\PP(\CW),\mathbb{Z}/\ell^m(a)) \leftarrow \oplus_{i=0}^{n-1} H^{b-2i}(\CX, \mathbb{Z}/\ell^m (a-i))\cdot z_m^i \notag}
is an isomorphism. This is  a local property on $\CX$, reduced by \cite[Lemma~4.4.1]{CS17} 
to the computation of the \'etale cohomology of $\PP(\mathcal{W)}$ over a geometric point 
$\bar x={\rm \mathrm{Spa}}(C(\bar x), C(\bar x)^{+})$, which then is the standard computation. 
\end{proof}

\begin{rmk}
{\rm We avoid here the delicate question whether the surjection
$$ H^{j}( \mathcal{X},\mathbb{Z}_\ell(i)) \to \varprojlim_m H^{j}( \mathcal{X},\mathbb{Z}/\ell^m(i))$$ 
is an isomorphism as it is irrelevant for our purpose.  }
\end{rmk}

Recall that we have level subgroups $K_{p,r} \subset G(\Qp)$. We denote by $\Sigma^H$ the invariants under a group $H$ acting on the set $\Sigma$.  We  use this notation for cohomology invariants.

\begin{prop}[Scholze]\label{etaleinvariant} 
The morphisms of ringed spaces 
$$_{K^p}\overline{\CS(G,X)}^{\rm min} \rightarrow  _{K^p \cdot K_{p,r}} \overline{\mathcal{S}(G,X)}^{\rm min} $$
 induce for all pairs of integers $(i,j)$  and all $r\ge 1$ homomorphisms
\ga{}{ 
H^j\left(_{K^p}\overline{\CS(G,X)}^{\rm min},\mathbb{Z}/\ell^m (i)\right)^{K_{p,r}}  \to   H^j\left(_{K^p\cdot K_{p,r}} S (G,X)^{\rm min}_{\overline{\mathbb{Q}}_p}, \mathbb{Z}/\ell^m(i)\right) \notag}
 and 
$$
\left[\varprojlim_m H^j\left({}_{K^p}\overline{\CS(G,X)}^{\rm min},\mathbb{Z}/\ell^m (i)\right)\right]^{K_{p,r}} =
  \varprojlim_m \left[H^j\left({}_{K^p}\overline{\CS(G,X)}^{\rm min},\mathbb{Z}/\ell^m (i)\right)\right]^{K_{p,r}} \!\!\!  \to   H^j\left({}_{K^p\cdot K_{p,r}} 
 S(G,X)^{\rm min}_{\overline{\mathbb{Q}}_p}, \mathbb{Z}_\ell(i)\right). 
$$
\end{prop}

\begin{proof} As before, we write $K_r = K^p\cdot K_{p,r}$.  By Theorem \ref{perfec2} and \cite[Corollary~7.18]{Sch12}, one has 
$$
H^j \left( _{K^p}\overline{\CS(G,X)}^{\rm min}, \mathbb{Z}/\ell ^m(i)\right)= \varinjlim_r H^j\left(
 {}_{K_r}\overline{\CS(G,X)}^{\rm min}, \mathbb{Z}/\ell ^m(i)\right)  =  \varinjlim_r H^j \left( 
_{K_r}\overline{S(G,X)}^{\rm min}_{\bar{\mathbb{Q}}_p}, \mathbb{Z}/\ell ^m(i)\right). 
$$
The second equality comes from the comparison isomorphism  
 \cite[Theorem~4.2]{Hub98}.

On the other hand, 
for $ r' > r \ge 1$ one has  
$K_r/K_{r'} = K_{p,r}/K_{p,r'}$, which is a $p$-group. 
It follows that the transition maps in the inductive system
 are injective as 
$ {}_{K_{r+1}}\overline{\CS(G,X)}^{\rm min} \to  {}_{K_r}\overline{\CS(G,X)}^{\rm min}$ is finite surjective of degree prime to $\ell$.  
This implies 
\begin{align}
H^j \left({} _{K^p}\overline{\CS(G,X)}^{\rm min}, \mathbb{Z}/\ell ^m(i)\right)^{K_{p,r}}  &=
 \varinjlim_{r'}  H^j \left({}_{K_{r'}}\overline{\CS(G,X)}^{\rm min}, \mathbb{Z}/\ell ^m(i)\right)^{K_{p,r}/K_{p,r'}} \notag \\
 &= \varinjlim_{r'}  H^j \left([_{K_{r'}}\overline{\CS(G,X)}^{\rm min}]/[K_{p,r}/K_{p,r'}], \mathbb{Z}/\ell ^m(i)\right)
\label{limm}
  \end{align}
  if $r \geq 1$, because $K_{p,r}/K_{p,r'}$ is a $p$-group.  Here $ [_{K_{r'}}\overline{\CS(G,X)}^{\rm min}]/[K_{p,r}/K_{p,r'}]$ 
  denotes the quotient of the (projective)  scheme $_{K_{r'}}\overline{\CS(G,X)}^{\rm min}$ by the finite $p$-group $K_{p,r}/K_{p,r'}$.

  If $r = 0$, the last two groups are equal up to finite error that is bounded independently of $r'$
  and $m$; we return
  to this case below.
  
For any $r' \geq r$, the quotient $_{K_{r'}}\overline{\CS(G,X)}^{\rm min}/[K_{p,r}/K_{p,r'}]$ lies in a sequence of finite morphisms
$$    _{K_{r}}\CS(G,X)^{\rm min}         \ra  _{K_{r'}}\overline{\CS(G,X)}^{\rm min}/[K_{p,r}/K_{p,r'}] \ra                        _{K_{r}}\overline{\CS(G,X)}^{\rm min}.$$
Indeed,  the minimal compactification is normal, which explains the left arrow, and the right one is just the image of the middle term in the  minimal compactification of the Siegel space at level $K_r$. 
 Define
$$_{K_{r}}\overline{\CS(G,X)}^{\rm min,\#} = \varprojlim_{r'} {}_{K_{r'}}\overline{\CS(G,X)}^{\rm min}/[K_{p,r}/K_{p,r'}]. $$
This is a scheme of finite type as it is dominated by  $ _{K_{r}}\CS(G,X)^{\rm min}   $. 
 Again we have a sequence of finite morphisms
$$    _{K_{r}}\CS(G,X)^{\rm min}         \ra  _{K_{r}}\overline{\CS(G,X)}^{\rm min,\#}  \ra     _{K_{r}}\overline{\CS(G,X)}^{\rm min}.$$
It follows from the above considerations that, for any $m$, $j$, and $i$,  we have
\begin{equation*}
 H^j \left({}_{K^p}\overline{\CS(G,X)}^{\rm min}, \mathbb{Z}/\ell ^m(i)\right)^{K_{p,r}}  =
  \varinjlim_{r}  H^j \left([_{K_{r}}\overline{\CS(G,X)}^{\rm min,\#}], \mathbb{Z}/\ell ^m(i)\right)
  \end{equation*}
  if $r \geq 1$, and with finite error independent of $m$ if $r = 0$.  
  
 So composing with the natural map 
 $$H^*\left({}_{K_r}\CS(G,X)^{\rm min,\#}, \mathbb{Z}/\ell ^m(i)\right) \ra H^*\left({}_{K_r}\CS(G,X)^{\rm min}, \mathbb{Z}/\ell ^m(i)\right) $$
  one obtains the homomorphism
\begin{equation}
\label{Kpr}
H^j \left({}_{K^p}\overline{\CS(G,X)}^{\rm min}, \mathbb{Z}/\ell ^m(i)\right)^{K_{p,r}} \to 
 H^j \left({}_{K_r}\CS(G,X)^{\rm min}, \mathbb{Z}/\ell ^m(i)\right) = H^j \left({}_{K_r} S(G,X)^{\rm min}_{\overline{\mathbb{Q}}_p}, \mathbb{Z}/\ell ^m(i)\right)
\end{equation}
 (see  again \cite{Hub98}, {\it loc.cit.}, for the comparison isomorphism),  if $r \geq 1$.
 
  Finally, by definition,  one has
\ga{}{\varprojlim_m H^j \left({}_{K^p}\CS(G,X)^{\rm min}, \mathbb{Z}/\ell ^m(i)\right)^{K_{p,r}}=
\left[\varprojlim_m H^j \left({}_{K^p}\CS(G,X)^{\rm min}, \mathbb{Z}/\ell ^m(i)\right)\right]^{K_{p,r}},\notag 
}
which thus maps to 
\ga{}{ \varprojlim_m H^j \left({}_{K_r} S(G,X)^{\rm min}_{\overline{\mathbb{Q}}_p}, \mathbb{Z}/\ell ^m(i)\right) = H^j \left({}_{K_r} S(G,X)^{\rm min}_{\overline{\mathbb{Q}}_p}, \mathbb{Z}_\ell(i)\right). \notag}

\end{proof}

 If $r = 0$ this remains true with $\mathbb{Z}/\ell ^m(i)$ replaced by $\mathbb{Q}_\ell(i)$:   the maps in \eqref{limm} are not necessarily isomorphisms but the kernels and cokernels are of order bounded independently of $m$ and $i$.


\bibliographymark{References}

\providecommand{\bysame}{\leavevmode\hbox to3em{\hrulefill}\thinspace}
\providecommand{\arXiv}[2][]{\href{https://arxiv.org/abs/#2}{arXiv:#1#2}}
\providecommand{\MR}{\relax\ifhmode\unskip\space\fi MR }
\providecommand{\MRhref}[2]{%
  \href{http://www.ams.org/mathscinet-getitem?mr=#1}{#2}
}
\providecommand{\href}[2]{#2}


\begin{thebibliography}{DK09-2}
\addcontentsline{toc}{section}{References}

\bibitem[AMRT75]{AMRT75} 
A. Ash, D. Mumford, M. Rapoport, and Y. Tai, \emph{Smooth compactification  of locally symmetric varieties.} In: Lie Groups: History, Frontiers and Applications, vol. IV, Math. Sci.  Press, Brookline,
Mass., 1975.
\MR{0457437}

\bibitem[BB66]{BB66} 
W. L. Baily, Jr. and A. Borel,
\emph{Compactification of arithmetic quotients of bounded symmetric domains},  
Ann. of Math. (2) {\bf 84} (1966), 442--528.
\MR{0216035}

\bibitem[BBD82]{BBD82} A. A. Beilinson, J. Bernstein, and P. Deligne,  
\emph{Faisceaux pervers}.
In: Analysis and topology on singular spaces, I (Luminy, 1981), pp. 5--171, 
Ast\'erisque, vol. 100, Soc. Math. France, Paris, 1982.
\MR{0751966}

\bibitem[Bha17]{Bha17} B. Bhatt,
\emph{Lecture notes for a class on perfectoid spaces}, 2017.  
\href{http://www-personal.umich.edu/~bhattb/teaching/mat679w17/lectures.pdf}{http://www-personal.umich.edu/ \textasciitilde bhattb/teaching/mat679w17/lectures.pdf}

\bibitem[BH58]{BH58}  A. Borel and F. Hirzebruch, \emph{Characteristic classes and homogeneous spaces. I}. Amer. J. Math. {\bf 80} (1958), 458--538.
\MR{0102800}

\bibitem[Bor83]{Bor83}  A. Borel, \emph{Regularization theorems in Lie algebra cohomology. Applications}. Duke Math. J. {\bf 50} (1983), no. 3, 605--623.
\MR{0714820}

\bibitem[CS17]{CS17}  A. Caraiani and P. Scholze, \emph{On the generic part of the cohomology of compact unitary Shimura varieties},  Ann. of Math. (2) {\bf 186} (2017), no. 3, 649--766.
\MR{3702677}

\bibitem[C+6]{C+6}  A. Caraiani, D. R. Gulotta, C.-Y. Hsu, C. Johansson, L. Mocz, E. Reinecke, and S.-C. Shih, \emph{Shimura varieties at level $\Gamma_1(p^\infty)$ and Galois representations},
preprint 2018. \arXiv{1804.00136}

\bibitem[dCM05]{dCM05} M. A. A. de Cataldo and L. Migliorini, 
\emph{Hodge theoretic aspects of the Decomposition Theorem}. In:
Algebraic Geometry -- Seattle 2005, part 2, pp. 489--504,
Proc. Sympos. Pure Math., vol. 80, part~2, Amer. Math. Soc., Providence, RI, 2009.
\MR{2483945}

\bibitem[Del71]{Del71} P. Deligne,
\emph{Travaux de Shimura}. In: 
S\'eminaire Bourbaki, 23\`eme ann\'ee (1970/71), Exp. No. 389, pp. 123--165,
Lectures Notes in Math., vol. 244, Springer, Berlin, 1971.
\MR{0498581}

\bibitem[Del74]{Del74} P. Deligne, \emph{Th\'eorie de Hodge. III}. 
Inst. Hautes \'Etudes Sci. Publ. Math. {\bf 44} (1974), 5--77.
\MR{0498552}

\bibitem[DMOS82]{DMOS82} P. Deligne, J. S. Milne, A. Ogus, and K. Shih, 
\emph{Hodge cycles,  motives and Shimura varieties},  
Lecture Notes in Mathematics, vol. 900, Springer-Verlag, Berlin-New York, 1982.
\MR{0654325}

\bibitem[EV02]{EV02} H. Esnault and E. Viehweg, 
\emph{Chern classes of Gauss-Manin bundles of weight 1 vanish},  
K-Theory {\bf 26} (2002), no. 3, 287--305.
\MR{1936356}

\bibitem[EH17]{EH17} H. Esnault and M. Harris, 
\emph{Chern classes of automorphic vector bundles}, 
Pure Appl. Math. Q. {\bf 13} (2017), no. 2, 193--213.
\MR{3858008}

\bibitem[GP02]{GP02} M. Goresky and W. Pardon,
\emph{Chern classes of automorphic vector bundles}, 
Invent. math. {\bf 147} (2002), no. 3, 561--612.
\MR{1893006}

\bibitem[GM03]{GM03} M. Goresky and R. MacPherson,
\emph{The topological trace formula}, 
J. Reine Angew. Math. {\bf 560} (2003), 77--150. 
\MR{1992803}

\bibitem[Har85]{Har85} M. Harris, 
\emph{Arithmetic vector bundles and automorphic forms on Shimura vartieties. I.} 
Invent. Math. {\bf 82} (1985), no. 1, 151--189.
\MR{0808114}

\bibitem[Har89]{Har89}  M. Harris,
\emph{Functorial properties of toroidal compactifications of locally symmetric varieties}, 
Proc. London Math. Soc. (3) {\bf 59} (1989), no. 1, 1--22. 
\MR{0997249}

\bibitem[Har90]{Har90}  M. Harris,
\emph{Automorphic forms of $\overline{\partial}$-cohomology type as coherent cohomology classes}, 
J. Differential Geom. {\bf 32} (1990), no. 1, 1--63.
\MR{1064864}

\bibitem[HZ94]{HZ94}  M. Harris and S. Zucker,
\emph{Boundary cohomology of Shimura varieties. I. Coherent cohomology on toroidal compactifications.} 
Ann. Sci. \'Ecole Norm. Sup. (4) {\bf 27} (1994), no. 3, 249--344. 
\MR{1272293}

\bibitem[Hub98]{Hub98} R. Huber,
\emph{A comparison theorem for $\ell$-adic cohomology}, 
Compositio Math. {\bf 112} (1998), no. 2, 217--235.
\MR{1626021}

\bibitem[Lan12]{Lan12} K.-W. Lan, 
\emph{Toroidal compactifications of PEL-type Kuga families}, 
Algebra Number Theory {\bf 6} (2012), no. 5, 885--966.
\MR{2968629}

\bibitem[Lan12a]{Lan12a} K.-W. Lan,
\emph{Comparison between analytic and algebraic constructions of toroidal compactifications of PEL-type Shimura varieties}, 
J. Reine Angew. Math. {\bf 664} (2012), 163--228. 
\MR{2980135}

\bibitem[Lan13]{Lan13} K.-W. Lan,
\emph{Arithmetic compactifications of PEL-type Shimura varieties}, 
London Mathematical Society Monographs, vol. 36, Princeton University Press, Princeton, NJ, 2013.
\MR{3186092}

\bibitem[Lan17]{Lan17}  K.-W. Lan,
\emph{Compactifications of PEL-type Shimura varieties and Kuga families with ordinary loci}, 
World Scientific Publishing Co. Pte. Ltd., Hackensack, NJ, 2018.
\MR{3729423}

\bibitem[Loo88]{Loo88} E. Looijenga, \emph{$L^2$-cohomology of locally symmetric varieties}, 
Compositio Math. {\bf 67} (1988), no. 1, 3--20. 
\MR{0949269}

\bibitem[Loo17]{Loo17} E. Looijenga, 
\emph{Goresky-Pardon lifts of Chern classes and some Tate extension},
Compositio Math. {\bf 153} (2017), no. 7, 1349--1371.
\MR{3705260}

\bibitem[Mum77]{Mum77}  D. Mumford,
\emph{Hirzebruch's proportionality theorem in the noncompact case},
Invent. math. {\bf 42} (1977), 239--272. 
\MR{0471627}

\bibitem[Nai14]{Nai14} A. N. Nair, 
\emph{Chern classes of automorphic bundles and the reductive Borel-Satake compactification},  
preprint 2014. \href{http://www.math.tifr.res.in/~arvind/preprints/chernrbs_v2.pdf}{http://www.math.tifr.res.in/\textasciitilde arvind/preprints/chernrbs\_v2.pdf}

\bibitem[PS16]{PS16} V. Pilloni and B. Stroh,
\emph{Cohomologie coh\'erente et repr\'esentations Galoisiennes}, 
Ann. Math. Qu\'e. {\bf 40} (2016), no. 1, 167--202.
\MR{3512528}

\bibitem[Pin90]{Pin90}  R. Pink,
\emph{Arithmetical compactification of mixed Shimura varieties},
Bonner Mathematische Schriften, vol. 209, Universit\"at Bonn, Mathematisches Institut, Bonn, 1990. 
\MR{1128753}

\bibitem[SS90]{SS90} L. Saper and M. Stern, 
\emph{$L_2$-cohomology of arithmetic varieties},  
Ann. of Math. (2) {\bf 132} (1990), no. 1, 1--69.
\MR{1059935}

\bibitem[Sch12]{Sch12} P. Scholze,
\emph{Perfectoid spaces}, 
Publ. Math. Inst. Hautes \'Etudes Sci. {\bf 116} (2012), 245--313.
\MR{3090258}

\bibitem[SW13]{SW13} P. Scholze and J. Weinstein,
\emph{Moduli of $p$-divisible groups}, Camb. J. Math. {\bf 1} (2013), no. 2, 145--237.
\MR{3272049}

\bibitem[Sch14]{Sch14} P. Scholze,
\emph{Perfectoid spaces and their applications}. In: Proceedings of the International Congress of Mathematicians -- Seoul 2014, vol. II, pp. 461--486, Kyung Moon Sa, Seoul, 2014. 
\MR{3728623}

\bibitem[Sch15]{Sch15} P. Scholze,
\emph{On torsion in the cohomology of locally symmetric varieties}, Ann. of Math. (2) {\bf 182} (2015),  no. 3, 945--1066. 
\MR{3418533}

\bibitem[SW17]{SW17}  P. Scholze and J. Weinstein,
\emph{Berkeley lectures on $p$-adic geometry}, 2014.
\href{http://www.math.uni-bonn.de/people/scholze/Berkeley.pdf}{http://www.math.uni-bonn.de/people/scholze/Berkeley.pdf}

\bibitem[Sch17]{Sch17} P. Scholze,
\emph{\'Etale cohomology of diamonds}, preprint 2017.
\arXiv{1709.07343}

\end{thebibliography}
\end{document}